\documentclass[12pt]{amsart}
\pdfoutput=1
\usepackage{graphicx}
\usepackage{amssymb}
\usepackage{amsmath}
\usepackage{amsthm}
\usepackage{amscd}
\usepackage{epsfig}

\newtheorem{theorem}{Theorem}[section]

\newtheorem{corollary}[theorem]{Corollary}
\newtheorem{lemma}[theorem]{Lemma}

\newtheorem{proposition}[theorem]{Proposition}

\newtheorem{remark}[theorem]{Remark}
\theoremstyle{definition}

\numberwithin{equation}{section}

\begin{document}

\title{Restrictions of Laplacian eigenfunctions to edges in the Sierpinski gasket}

\author [Hua Qiu and Haoran Tian]{Hua Qiu and Haoran Tian}

\address{Department of Mathematics, Nanjing University,
Nanjing, Jiangsu, 210093, P.R. China} \email{huaqiu@nju.edu.cn}

\address{Department of Mathematics, Nanjing University,
Nanjing, Jiangsu, 210093, P.R. China} \email{hrtian@hotmail.com}



\thanks{The research of the first author was supported by the National Science Foundation of China, Grant 11471157.}

\begin{abstract}
In this paper, we study the restrictions of both the harmonic functions and the eigenfunctions of the symmetric Laplacian to edges of pre-gaskets contained in the Sierpinski gasket. For a harmonic function, its restriction to any edge is either monotone or having a single extremum. For an eigenfunction, it may have several local extrema along edges. We prove general criteria, involving the values of any given function at the endpoints and midpoint of any edge, to determine which case it should be, as well as the asymptotic behavior of the restriction near the endpoints. Moreover, for eigenfunctions, we use spectral decimation to calculate the exact number of the local extrema along any edge. This confirm, in a more general situation, a conjecture of K. Dalrymple, R.S. Strichartz and J.P. Vinson \cite{DSV} on the behavior of the restrictions to edges of the basis Dirichlet eigenfunctions, suggested by the numerical data. \end{abstract}

\maketitle

\section{Introduction}

The study of analysis on  \textit{post critically finite (p.c.f.) self-similar sets} has been extensively developed since Kigami's analytic construction of the fractal Laplacian on the \textit{Sierpinski gasket} \cite{K1,K2}. Such a Laplacian  has also been obtained independently by Goldstein \cite{G} and Kusuoka \cite{Ku} using an indirect probabilistic approach. The harmonic functions and Laplacian eigenfunctions play important role in the analysis on the Sierpinski gasket, as the Laplacian $\Delta$ is the fundamental differential operator on which the analysis is based. 
Harmonic functions are the solutions of $\Delta f=0$, and Laplacian eigenfunctions are non-zero functions satisfying $-\Delta f=\lambda f$ for some eigenvalue $\lambda$. Fukushima and Shima \cite{FS, Sh} studied the spectrum of the Dirichlet Laplacian, whose eigenfunctions vanish on the boundary, and determined explicitly all the eigenvalues and eigenfunctions by using a \textit{spectral decimation} method introduced by physicists Rammal and Toulouse \cite{RT}. See \cite{BK, K3, KL, T} and the reference therein for related works on the eigenvalues and eigenfunctions. In this paper, we consider the local structure of these special functions on the Sierpinski gasket. More precisely, for any edge of the pre-gaskets of the Sierpinski gasket, we are interested in the local behavior of the restrictions of these special functions. 

Recall the Sierpinski gasket, denoted by $\mathcal{SG}$, is the attractor of the \textit{iterated function system (i.f.s.)} in $\mathbb{R}^2$ consisting of three contractive mappings $F_0, F_1, F_2$ with contraction ratio $\frac{1}{2}$ and fixed points $q_0, q_1, q_2$ which are the three vertices of an equilateral triangle. So
\begin{equation*}
\mathcal{SG}=\bigcup_{i=0}^2 F_i\mathcal{SG},
\end{equation*}
see Fig. 1. We denote by $V_0=\{q_0,q_1,q_2\}$ the \textit{boundary} of $\mathcal{SG}$ and $V_m=\bigcup_{i=0}^{2}F_i V_{m-1}$ inductively. Write $V_*=\bigcup_{m\geq 0}V_m$. For a \textit{word} $w=w_1\cdots w_m$ of \textit{length} $|w|=m$ with $w_i\in\{0,1,2\}$, let $F_w=F_{w_1}\circ\cdots\circ F_{w_m}$, and call $F_w\mathcal{SG}$ a \textit{$m$-cell} of $\mathcal{SG}$. Obviously, $V_m=\bigcup_{|w|=m} F_wV_0$. We approximate $\mathcal{SG}$ from within by a sequence of \textit{graphs} $\Gamma_m=(V_m, \sim_m)$ with vertices $V_m$ and edge relation $\sim_m$ defined as $x\sim_m y$ if and only if there is a word $w$ of length $m$ such that $x\neq y$ and $x,y\in F_w V_0$.

\begin{figure}[h]
\begin{center}
\includegraphics[width=6.7cm]{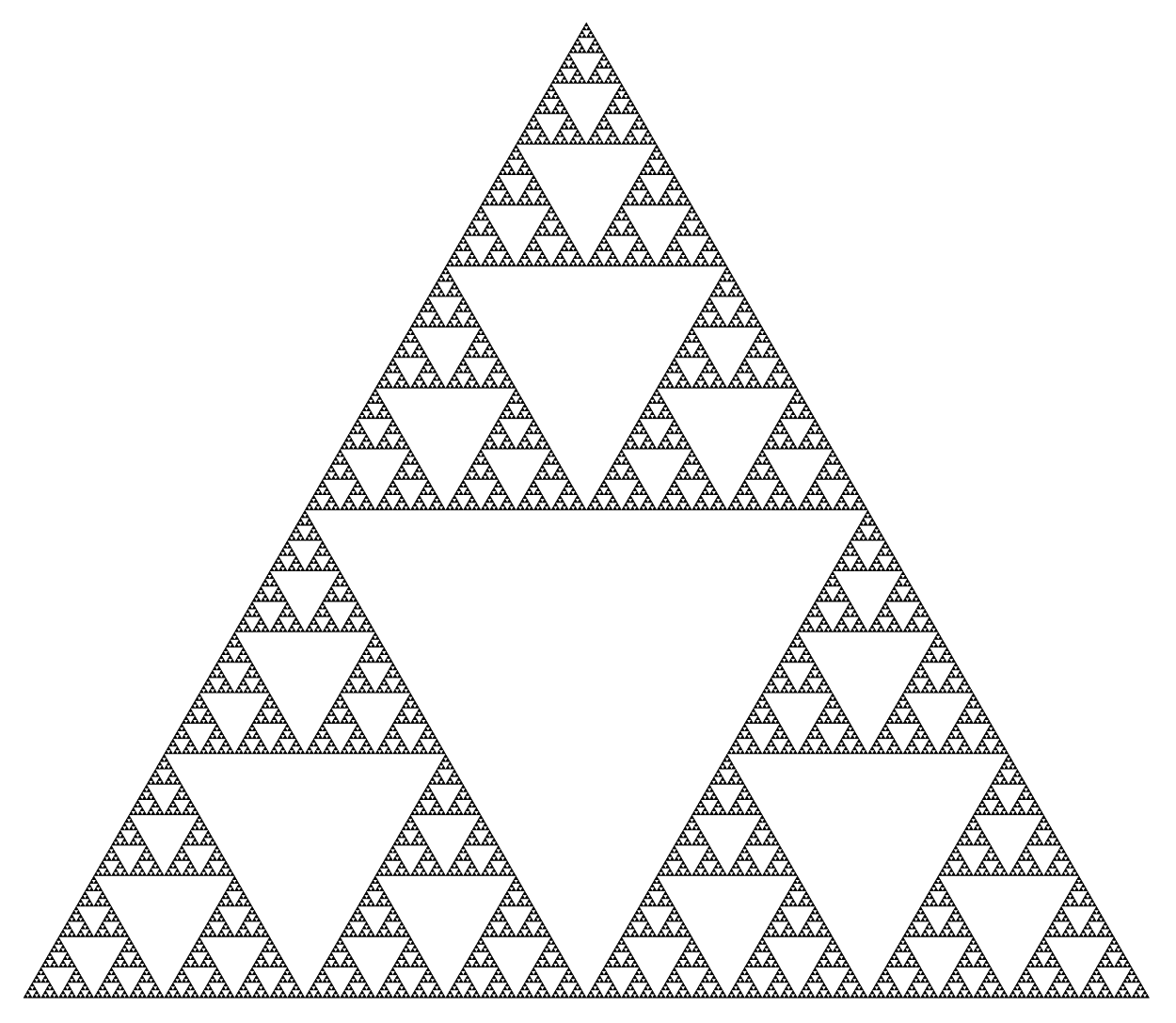}
\begin{center}
\begin{picture}(0,0) \thicklines
\end{picture}

\textbf{Figure 1. The Sierpinski gasket}
\end{center}\end{center}
\end{figure}

The \textit{(symmetric) Laplacian} $\Delta$ on $\mathcal{SG}$ is the renormalized limit of \textit{graph Laplacians} $\Delta_m$ on $\Gamma_m$ defined as:
\begin{equation*}
\Delta_m u(x)=\sum_{y\sim_m x} \big(u(y)-u(x)\big) \text{ for } x\in V_m\setminus V_0,
\end{equation*}
and
\begin{equation}\label{eq1}
\Delta u(x)=\frac{3}{2}\lim_{m\rightarrow\infty}5^m\Delta_m u(x)
\end{equation}
with $u\in dom\Delta$ if and only if the right side of (\ref{eq1}) converges uniformly on $V_*\setminus V_0$ to a continuous function. 

A function $h$ is called \textit{harmonic} if it minimize the graph energy $\sum_{x\sim_m y}\big(h(x)-h(y)\big)^2$ on $\Gamma_m$ from each level to its next level. The space of harmonic functions on $\mathcal{SG}$ is $3$-dimensional and the values at vertices in $V_0$ can be freely assigned. The harmonic functions are also graph harmonic, i.e., they are the solutions of $\Delta_m h=0$ for each $m\geq 1$. So it is elementary to calculate the values of $h$ on $V_1$ from those on $V_0$ and recursively to obtain the values on $V_m$ for any $m$. There is an explicit \textit{extension algorithm}, the ``$\frac{1}{5}-\frac{2}{5}$'' rule, for computing the values of harmonic functions at all vertices in $V_*$ in terms of the boundary values. That is for any cell in $\mathcal{SG}$ whose boundary vertices are denoted by $p_0,p_1,p_2$, we have $h(p_{ij})=\frac{2}{5}h(p_i)+\frac{2}{5}h(p_j)+\frac{1}{5}h(p_k)$, where  $p_{ij}$ is the midpoint of the edge joining $p_i$ and $p_j$, for distinct $i,j,k\in\{0,1,2\}$.

Similar to that of the harmonic functions, for an \textit{eigenfunction} $u$ satisfying $-\Delta u=\lambda u$, there is a minimal integer $m_0$, called the \textit{level of birth}, such that  for $m\geq m_0$, the restriction of $u$ to $V_m$ is also an eigenfunction of the graph Laplacian associated with an eigenvalue $\lambda_m$, i.e., $-\Delta_m u=\lambda_m u$, satisfying that
\begin{equation}\label{eq2}
\lambda_{m}=\lambda_{m+1}(5-\lambda_{m+1})
\end{equation}
and 
\begin{equation}\label{eq3}
\lambda=\frac{3}{2}\lim_{m\rightarrow\infty}5^m\lambda_m.
\end{equation}
The identity (\ref{eq2}) implies that 
\begin{equation*}
\lambda_{m+1}=\frac{5+\varepsilon_{m+1}\sqrt{25-4\lambda_{m}}}{2} \quad\text{ for } \varepsilon_{m+1}=\pm 1, m\geq m_0.
\end{equation*}
Call $\{\varepsilon_{m}\}_{m>m_0}$ the \textit{$\varepsilon$-sequence} of $u$ and $\lambda$. Note that to ensure the limit in (\ref{eq3}) exists, $\varepsilon_m$ is only permitted to be positive $1$ for finitely many values of $m$. So there is an minimal integer $m_1> m_0$, called the \textit{level of fixation} such that for each $m\geq m_1$, $\varepsilon_{m}=-1$. There is a \textit{local extension algorithm} similar to that of the harmonic functions, which extends a graph eigenfunction $u$ on $V_m$ associated with $\lambda_m$ to a graph eigenfunction on $V_{m+1}$ associated with $\lambda_{m+1}$, provided that $\lambda_{m+1}\neq 2,5$. For any $m$-cell whose boundary vertices are $p_0, p_1, p_2$, we have
\begin{equation}\label{eq4}
u(p_{ij})=\frac{(4-\lambda_{m+1})\big(u(p_i)+u(p_j)\big)+2u(p_k)}{(2-\lambda_{m+1})(5-\lambda_{m+1})}
\end{equation}
with $p_{ij}$ being the vertex in $V_{m+1}$ in that cell, opposite to $p_k$, for distinct $i,j,k\in\{0,1,2\}$.

We summarize this spectral decimation property into the following proposition, see also Lemma 3.2.1 in \cite{St6}. 
\begin{proposition}
Suppose $\lambda_{m+1}\neq 2,5,6$ and $\lambda_m$ is given by (\ref{eq2}). If $u|_{V_m}$ is a $\lambda_m$-eigenfunction of $-\Delta_m$ and is extended to $V_{m+1}$ by (\ref{eq4}), then $u|_{V_{m+1}}$ is a $\lambda_{m+1}$-eigenfunction of $-\Delta_{m+1}$. Conversely, if $u|_{V_{m+1}}$ is a $\lambda_{m+1}$-eigenfunction of $-\Delta_{m+1}$, then $u|_{V_m}$ is a $\lambda_m$-eigenfunction of $-\Delta_m$.
\end{proposition}
 This property was proved by Fukushima and Shima \cite{FS, Sh} mathematically. It is not true for all p.c.f. self-similar sets, even for other non-symmetric Laplacians on $\mathcal{SG}$. See \cite{Sh2} for the extension of this property to some fully symmetric fractals. 
 
 For the Dirichlet eigenfunctions, it is proved that we can classify them into three families, which we called \textit{$2$-series}, \textit{$5$-series} and \textit{$6$-series} eigenfunctions, depending on the initial eigenvalue $\lambda_{m_0}=2,5$ or $6$. The  $2$-series eigenfunctions all have $m_0=1$, the $5$-series eigenfunctions have $m_0\geq 1$, and the 
$6$-series eigenfunctions have $m_0\geq 2$ together with $\lambda_{m_0+1}=3$ ($\varepsilon_{m_0+1}=1$). The $2$-series eigenvalues all have multiplicity $1$, while the eigenvalues in the other series all exhibit higher multiplicity. The case for the Neumann eigenfunctions is very similar. See Section 3.3 in \cite{St6} for a complete description of the Dirichlet (Neumann) spectrum for the $\mathcal{SG}$ and a basis (not orthonormal) for all Dirichlet (Neumann) eigenfunctions.  Here we only mention that all such basis functions are formed from gluing of scaled and rotated copies of functions as shown in Fig. 2.

\begin{figure}[h]
\begin{center}
\includegraphics[width=4.5cm]{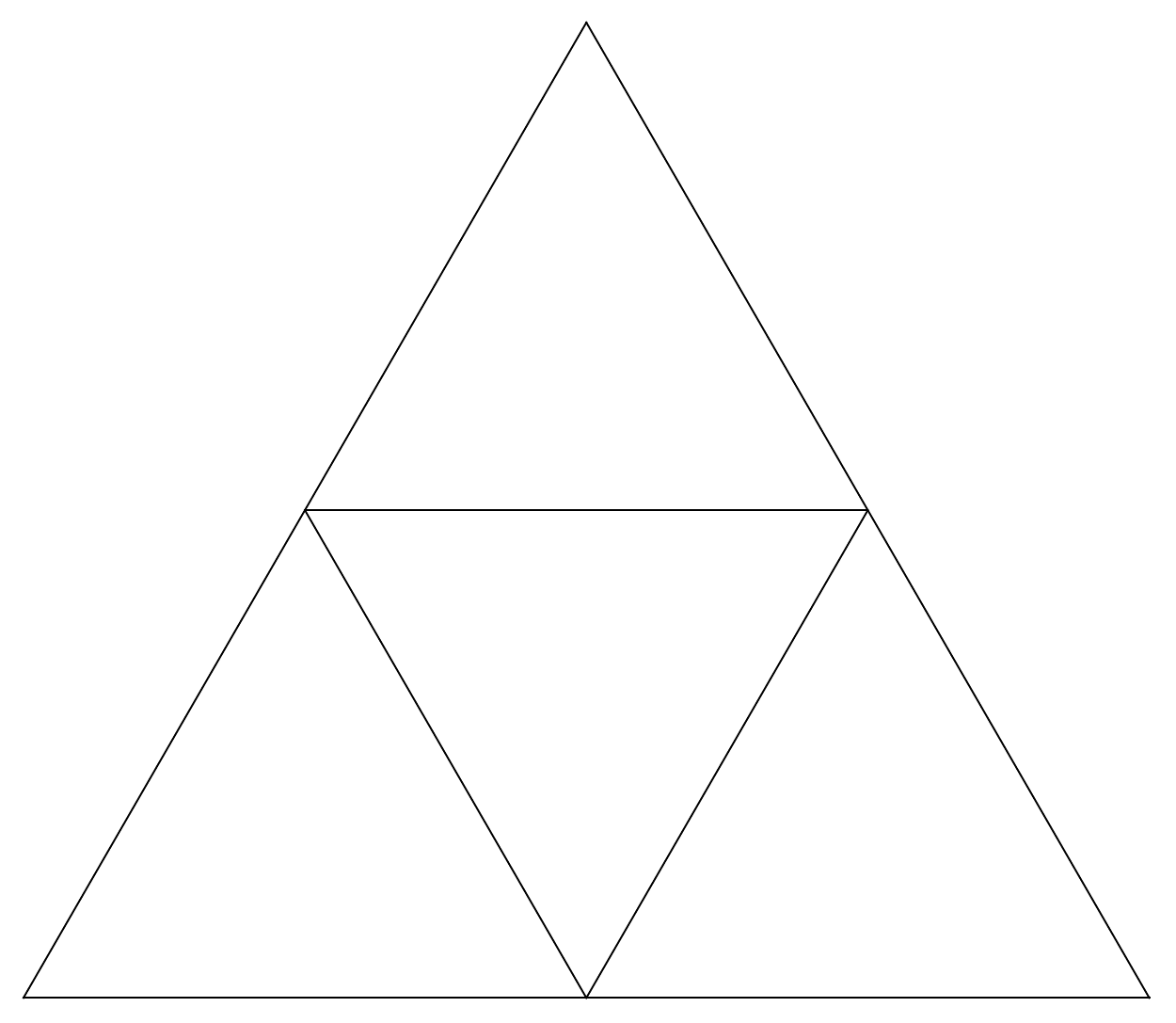}\hspace{0.3cm}
\includegraphics[width=4.5cm]{graph1.pdf}\hspace{0.3cm}
\includegraphics[width=4.5cm]{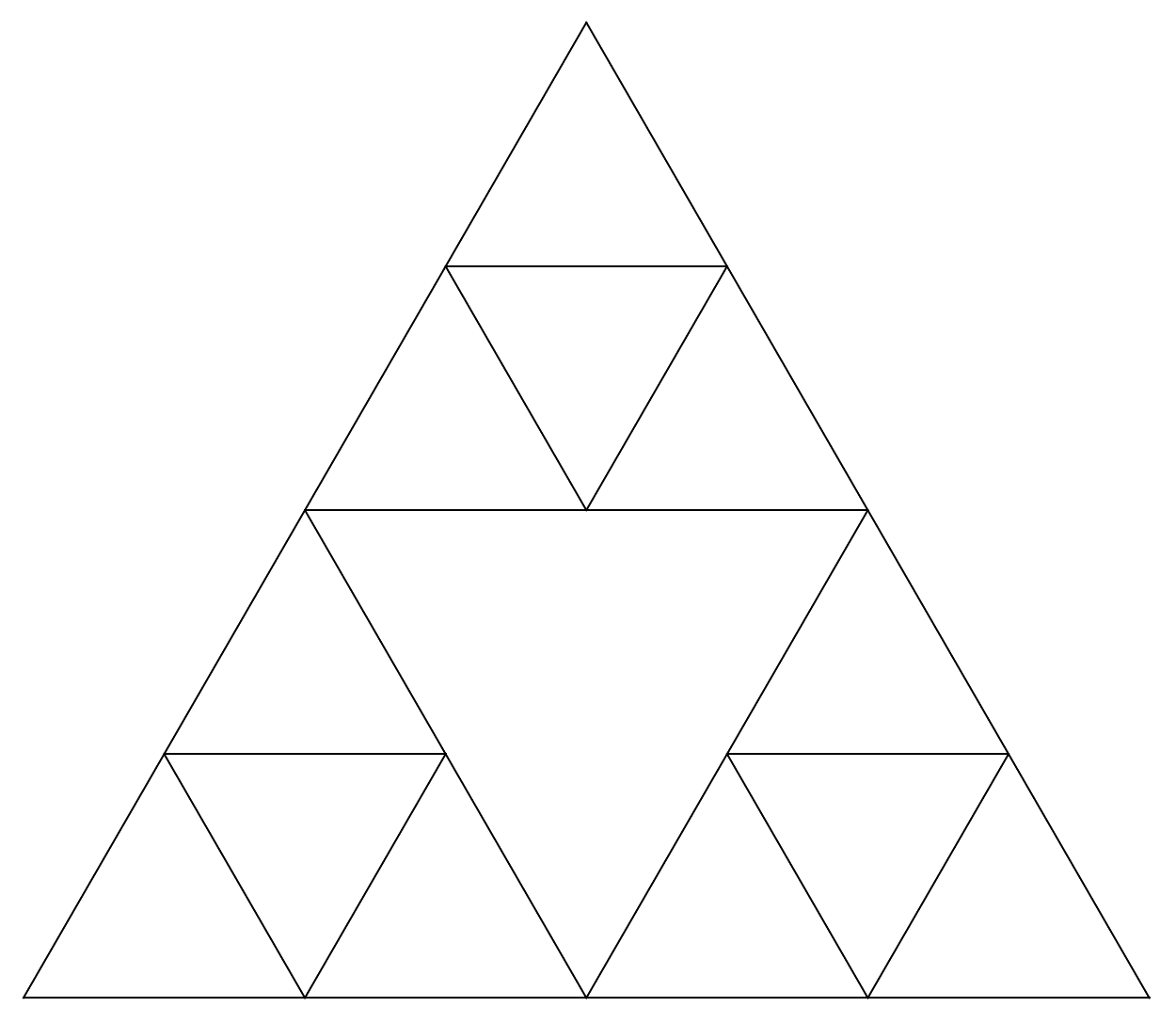}
\begin{center}
\begin{picture}(0,0) \thicklines
\put(-180,70){$1$}
\put(-143,8){$1$}
\put(-108,70){$1$}
\put(-206,8){$0$}
\put(-80,8){$0$}\put(101,8){$-1$}
\put(-143,127){$0$}\put(163,8){$-1$}
\put(-3,127){$0$}\put(138,127){$0$}\put(158,95){$0$}
\put(-3,8){$0$}\put(200,8){$0$}\put(118,95){$0$}
\put(-66,8){$0$}\put(75,8){$0$}\put(138,63){$0$}
\put(60,8){$0$}\put(138,8){$2$}\put(187,44){$1$}
\put(31,70){$-1$}\put(172,70){$0$}\put(141,44){$-1$}
\put(-38,70){$1$}\put(102,70){$0$}\put(87,44){$1$}\put(123,44){$-1$}
\put(-156,-5){$\lambda_1=2$}
\put(-16,-5){$\lambda_1=5$}
\put(124,-5){$\lambda_2=6$}
\end{picture}

\textbf{\\Figure 2. Three typical Dirichlet eigenfunctions with $\lambda_1=2$ or $5$ and $\lambda_2=6$}
\end{center}\end{center}
\end{figure}
 The reader is referred to the books \cite{K4} and \cite{St6} for exact definitions and any unexplained notations. 

In \cite{DSV}, it was proved that the restriction of any harmonic function along edges is either monotone or having a single extremum. And from the numerical computations shown in \cite{DSV}, it appears that there is a much more complicated structure for Laplacian eigenfunctions. The restrictions sometimes are monotone, and sometimes may attain several local extrema. Moreover, for a graph eigenfunction on $\Gamma_{m}$ for some $m$, if we index all the eigenfunctions, initialed from it and continued by the spectral decimation (choosing all but a finite number of $\varepsilon_{m'}=-1$ with $m'\geq {m+1}$), by a natural number $n$ so that the corresponding eigenvalue is increasing of $n$, then the exact number of the local extrema along edges seems to satisfy some pattern which is a function of $n$. In particular, let's start from the three typical Dirichlet eigenfunctions shown in Fig. 2 to get three basic families of eigenfunctions, denoted by $\{\psi_n^{(2)}\}_{n\geq 0}$, $\{\psi_n^{(5)}\}_{n\geq 0}$ and $\{\psi_n^{6}\}_{n\geq 0}$, with $\psi_0^{(2)}$, $\psi_0^{(5)}$ choosing all $\varepsilon_m=-1$ for $m\geq 2$, and $\phi_{0}^{(6)}$ choosing $\varepsilon_{3}=1$ and all $\varepsilon_m=-1$ for $m\geq 4$. Then it was conjectured in \cite{DSV} that the numbers of the local extrema of the restrictions of $\psi_n^{(2)}$, $\psi_n^{(5)}$ and $\psi_n^{(6)}$ to the bottom edge of $\mathcal{SG}$ are $2n+1$, $4[\frac{n}{2}]+2$ and $8n+3$ respectively, and the number of local extrema of the function $\psi_n^{(5)}$ restricted to the left edge of $\mathcal{SG}$ should be $2[\frac{n}{2}]+n+1$,  where we use $[x]$ to denote the largest integer no larger than $x$.

Let $E$ be an edge of a $m$-cell in $\mathcal{SG}$, call it a \textit{$m$-edge}. Denote the endpoint of $E$ by $p_0^E$, $p_1^E$ and the midpoint by $p_{01}^E$. For any function $f$ which is non-constant on $E$, denote
\begin{equation*}
r^E(f)=\frac{f(p_1^E)-f(p_{01}^E)}{f(p_{01}^E)-f(p_0^E)}
\end{equation*}
\big(allow $r^E(f)$ to be $\infty$, the point at infinity, not distinguishing between $+\infty$ and $-\infty$, when $f(p_{01}^E)=f(p_0^E)$\big).
We will prove criteria involving the parameter $r^E(f)$ to determine whether the restriction of $f$ to $E$ is monotone or not. More precisely, denote $\phi_i^E: \mathbb{R}^2\rightarrow\mathbb{R}^2$ the contractive mapping by $\phi_i^E(x)=\frac{1}{2}(x-p_i^E)+p_i^E$ for $i=0,1$, and write $E_\tau=\phi^E_\tau E:=\phi^E_{\tau_1}\circ\cdots\phi^E_{\tau_l}E$ for any word $\tau=\tau_{1}\cdots\tau_l$ of length $|\tau|=l$, with $\tau_i=0$ or $1$. Obivously, $E_\tau$ is a $(m+l)$-edge in $\mathcal{SG}$.
We have:
\begin{theorem}\label{thm1}
Let $h$ be a non-constant harmonic function and $E$ be a $m$-edge in $\mathcal{SG}$. Then

(1) if $\frac{1}{4}\leq r^E(h)\leq 4$, then $h|_{E}$ is strictly monotone;

(2) otherwise $h|_{E}$ has a single extremum. 

Moreover, if we denote by 
\begin{equation}\label{eq5}
\alpha^E(h)=\lim_{l\rightarrow\infty}\frac{|h(p_1^{E_{0^l}})-h(p_0^E)|}{|p_1^{E_{0^l}}-p_0^E|} \text{ and }\beta^E(h)=\lim_{l\rightarrow\infty}\frac{|h(p_0^{E_{1^l}})-h(p_1^E)|}{|p_0^{E_{1^l}}-p_1^{E}|},
\end{equation}
 then $\alpha^E(h)=0$ when $r^E(h)=4$, $\beta^E(h)=0$ when $r^E(h)=\frac{1}{4}$,  and otherwise $\alpha^E(h)=\infty$, $\beta^E(h)=\infty$.
\end{theorem}

\begin{theorem}\label{thm2}
Let $u$ be a Dirichlet (or Neumann) eigenfunction and $E$ be a $m$-edge in $\mathcal{SG}$ with $m\geq m_1-1$. Suppose $u$ is non-constant on $E$, then

(1) if $\frac{1}{4-\lambda_{m+1}}\leq r^E(u)\leq 4-\lambda_{m+1}$, then $u|_{E}$ is strictly monotone;

(2) otherwise $u|_{E}$ has a single extremum. 

Moreover, we have $\alpha^E(u)=0$ when $r^E(u)=4-\lambda_{m+1}$, $\beta^E(u)=0$ when $r^E(u)=\frac{1}{4-\lambda_{m+1}}$,  and otherwise $\alpha^E(u)=\infty$, $\beta^E(u)=\infty$, where $\alpha^E(u)$, $\beta^E(u)$ are defined similarly to that in (\ref{eq5}) with $h$ replaced by $u$.
\end{theorem}

We remark here that  part (1) in Thoerem \ref{thm1} is exact Theorem 1 in \cite{DSV}, and the values of $\alpha^E$ and $\beta^E$ reflect the asymptotic behavior of $h|_E$ (or $u|_{E}$) near the boundary of the edge $E$. The assumption $u$ is non-constant along $E$ in Theorem \ref{thm2} is necessary since there indeed exist eigenfunctions which are constant along certain edges, see Theorem 5 in \cite{DSV}. The result in Theorem \ref{thm2} also holds for general eigenfunctions when $m$ is sufficiently large since $\lambda_m\rightarrow 0$ as $m$ goes to $\infty$.

Note that the assumption $m\geq m_1-1$ in Theorem \ref{thm2} makes all $\varepsilon_{m+i}=-1$ for $i\geq 1$. Without this assumption, things will be much more complicated. In fact, when $m<m_1-1$, in general, the restriction of $u$ to $E$ will attain several local extrema. The following theorem calculate the exact number of extrema that occur in this case. 

\begin{theorem}\label{thm3}
Let $u$ be a Dirichlet (or Neumann) eigenfunction and $E$ be a $m$-edge in $\mathcal{SG}$ with $m\geq m_0$ (We require $m>m_0$ if $u$ is a $6$-series eigenfunction.). Suppose $u$ is non-constant on $E$, and let $\{\psi_n\}_{n\geq 0}$ be the sequence of all eigenfunctions initialed from $u|_{V_m}$ with the ordering making the corresponding eigenvalue an  increasing function of $n$. Then
 \begin{equation*}
      N^E(\psi_n)=
      \begin{cases}
     2[\frac{n+1}{2}],\quad &\text{if } \frac{1}{4-\lambda^{(0)}_{m+1}}<r^E(\psi_0)<4-\lambda^{(0)}_{m+1},\\
   n, \quad &\text{if } r^E(\psi_0)=\frac{1}{4-\lambda^{(0)}_{m+1}}\text{ or }4-\lambda^{(0)}_{m+1},\\
      2[\frac{n}{2}]+1,  &\text{otherwise},\\
      \end{cases}
      \end{equation*}
      where $N^E(\psi_n)$ denotes the number of local extrema of $\psi_n$ in $E$, and $\lambda^{(0)}_{m+1}$ is the $(m+1)$-th graph eigenvalue of $\psi_0$, i.e., $\lambda^{(0)}_{m+1}=\frac{5-\sqrt{25-4\lambda_m}}{2}$.
\end{theorem}

Note that Theorem \ref{thm3} is a natural extension of Theorem \ref{thm2}, since in case of Theorem \ref{thm2}, $u$ is $\psi_0$. By using Theorem \ref{thm3}, with a bit of more work, we can check that the conjecture in \cite{DSV} which we mentioned before is true. That is

\begin{theorem}\label{thm4}
Let $\{\psi_n^{(2)}\}_{n\geq 0}$, $\{\psi_n^{(5)}\}_{n\geq 0}$ and $\{\psi_n^{(6)}\}_{n\geq 0}$ be the three basic families of eigenfunctions initialed from the graph eigenfunctions in Fig. 2. The numbers of local extrema along the bottom edge of $\mathcal{SG}$ of these functions are $2n+1$, $4[\frac{n}{2}]+2$ and $8n+3$ respectively, and the number of local extrema of $\psi_n^{(5)}$ along the left edge of $\mathcal{SG}$ is $2[\frac{n}{2}]+n+1$.
\end{theorem}

The paper is organized as follows. In Section 2, we will deal with the restrictions of harmonic functions to edges and prove Theorem \ref{thm1}.  In Section 3, we come to the Laplacian eigenfunctions and prove Theorem \ref{thm2} and \ref{thm3}. As an application of Theorem \ref{thm3}, we will prove Theorem \ref{thm4} in Section 4. There are a number of interesting features observable in the numerical computations of harmonic functions and Laplacian eigenfunctions for $\mathcal{SG}$. In this paper we have deliberately restricted our attention to the local behavior of these special functions restricted to edges. See \cite{DRS} for a related result concerning the gradients of the eigenfunctions on the Sierpinski gasket. 

\section{restrictions of harmonic functions to edges}

In this section, we study the local behavior of the restrictions of harmonic functions to edges. We will prove Theorem \ref{thm1}. 

Let $h$ be a non-constant harmonic function and $E$ be a $m$-edge in $\mathcal{SG}$. It is known that $h$ is non-constant on $\mathcal{SG}$ if and only if $h$ is non-constant on $E$. We denote $p_0^E$, $p_1^E$, $p_{01}^E$, $r^E(h)$, and $E_\tau$ as we did in Section 1, where $\tau$ is any word $\tau=\tau_1\cdots\tau_l$ with $\tau_i=0$ or $1$ and $|\tau|=l\geq 0$. Without causing any confusion, we will remove the superscript $E$ for simplicity, and write $r$, $r^\tau$, $p_0^\tau$, $p_1^\tau$, $p_{01}^\tau$ instead of $r^E(h)$, $r^{E_\tau}(h)$, $p_0^{E\tau}$, $p_1^{E_\tau}$, $p_{01}^{E_\tau}$ for short. Obviously, for any word $\tau$, $p_{01}^\tau=p_0^{\tau 1}=p_1^{\tau 0}$, and 
\begin{equation*}
r=\frac{h(p_1)-h(p_{01})}{h(p_{01})-h(p_0)}
\text{ and }
r^\tau=\frac{h(p_1^\tau)-h(p_{01}^\tau)}{h(p_{01}^\tau)-h(p_0^\tau)},
\end{equation*}
where we allow $r$, $r^\tau$ to take $\infty$, the point at infinity, gluing $\pm \infty$ together as a single point, when $h(p_{01})=h(p_0)$ or $h(p_{01}^\tau)=h(p_0^\tau)$. If we denote $p_2$ the other vertex in the $m$-cell containing $E$, then we also have
\begin{equation*}
r=\frac{3h(p_1)-2h(p_0)-h(p_2)}{2h(p_1)+h(p_2)-3h(p_0)}
\end{equation*}
by using the ``$\frac{1}{5}-\frac{2}{5}$'' extension rule.

There is an algorithm which enables us to calculate the restriction of a harmonic function to any edge in $\mathcal{SG}$. Let $p_{001}^\tau$ be the midpoint of the edge connecting $p_0^\tau$ and $p_{01}^\tau$. Then
\begin{equation}\label{eq7}
h(p_{001}^\tau)=\frac{4}{5}h(p_{01}^\tau)+\frac{8}{25}h(p_{0}^\tau)-\frac{3}{25}h(p_1^\tau).
\end{equation}
See Algorithm 2.2 in \cite{DSV} for  a proof of (\ref{eq7}) by using the ``$\frac{1}{5}-\frac{2}{5}$'' rule.

\begin{lemma}\label{lemma1}
For any word $\tau$, we have 
\begin{equation}\label{eq6}
r^{\tau 0}=\frac{8+3r^\tau}{17-3 r^\tau}
\text{ and }r^{\tau1}=\frac{17r^\tau-3}{8r^\tau+3},
\end{equation} 
where the two mappings in (\ref{eq6}) are both bijections: $\mathbb{R}\cup\{\infty\}\rightarrow\mathbb{R}\cup\{\infty\}$. In particular, $r^{\tau 0}=-1$, $r^{\tau 1}=\frac{17}{8}$ when $r^\tau=\infty$. See Fig. 3 for the graphs of $r^{\tau 0}$ and $r^{\tau1}$ of $r^\tau$.
\end{lemma}

\begin{figure}[h]
\begin{center}
\includegraphics[width=11cm]{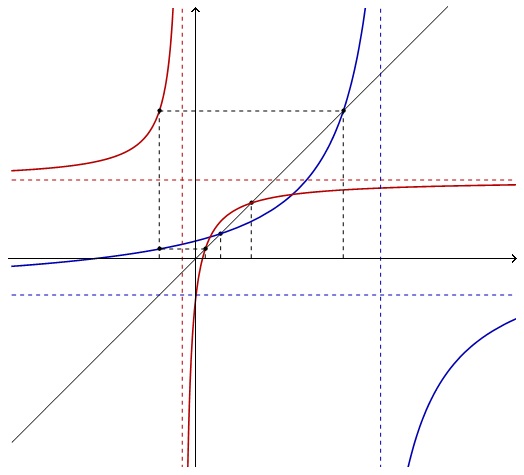}
\begin{center}
\begin{picture}(0,0) \thicklines
\put(154,145){$r^\tau$}
\put(124,180){$r^{\tau1}$}
\put(134,90){$r^{\tau0}$}
\put(43,280){$r^{\tau0}$}
\put(-71,280){$r^{\tau1}$}
\put(-70,139){\tiny$-1$}\put(-59,138){\tiny$-\frac38$}\put(-36,138){\tiny$\frac14$}\put(-27,138){\tiny$\frac23$}
\put(-9,138){\tiny$\frac32$}\put(47,138){\tiny$4$}\put(71,138){\tiny$\frac{17}{3}$}
\put(-40,196){\tiny$\frac{17}{8}$}\put(-40,118){\tiny$-1$}\put(-38,236){\tiny$4$}

\end{picture}
\textbf{\\Figure 3. Graphs of $r^{\tau0}$ and $r^{\tau1}$ of $r^\tau$}
\end{center}\end{center}
\end{figure}

\begin{proof}
The first equality in (\ref{eq6}) follows from (\ref{eq7}) by an elementary calculation. The second equality follows similarly by a symmetric consideration. It can be directly checked that the two mappings are bijections.
\end{proof}

\begin{remark}\label{re1}
It is easy to check from Fig. 3 that the mapping $f(x)=\frac{8+3x}{17-3x}$ from $\mathbb{R}\cup\{\infty\}$ onto $\mathbb{R}\cup\{\infty\}$ has a single stable fixed point $x=\frac{2}{3}$ and a single unstable fixed point $x=4$, which means that for any $x\neq 4$, the iterated sequence $x, f(x), f(f(x)),\cdots$ will go to $\frac{2}{3}$. Similarly, the mapping $f(x)=\frac{17x-3}{8x+3}$ has a single stable fixed point $x=\frac{3}{2}$ and a single unstable fixed point $x=\frac{1}{4}$. 
\end{remark}

To prove Theorem \ref{thm1}, the following lemma is useful.

\begin{lemma}\label{lemma2}
For any word $\tau$, the restriction of $h$ to $E$ can not attain an extremum at $p_{01}^\tau$ if $r^\tau\neq -1$.
\end{lemma}
\begin{proof}

Without loss of generality, we only need to prove that $h$ can not attain an extremum at $p_{01}$ if $r\neq -1$. Suppose $r\neq -1$, then $h(p_0)\neq h(p_1)$. We divide the restriction of $h$ to $E$ into its symmetric part $h_S$ and antisymmetric part $h_A$. That is,
\begin{eqnarray*}
&&h_S(p_0)=h_S(p_1)=\frac{1}{2}\big(h(p_0)+h(p_1)\big),\quad  h_S(p_{01})=h(p_{01}),\\
&&h_A(p_0)=-h_A(p_1)=\frac{1}{2}\big(h(p_0)-h(p_1)\big),\quad  h_A(p_{01})=0,
\end{eqnarray*}
and both $h_S$ and $h_A$ satisfy (\ref{eq7}).
Then we have $h=h_S+h_A$ on $E$, and $r(h_S)=-1$( without loss of generality, assuming that $h_S$ is not a constant function), $r(h_A)=1$. Notice that $h_A$ can not attain an extremum at $p_{01}^\tau$ from its antisymmetry.

Now we consider the asymptotic behavior of $h_S$ and $h_A$ near the point $p_{01}$. By Lemma \ref{lemma1} and Remark \ref{re1}, $\lim_{l\rightarrow\infty}r^{ 10^l}(h_S)=4$,  $\lim_{l\rightarrow\infty}r^{01^l}(h_S)=\frac{1}{4}$,  $\lim_{l\rightarrow\infty}r^{ 10^l}(h_A)=\frac{2}{3}$, and $\lim_{l\rightarrow\infty}r^{01^l}(h_A)=\frac{3}{2}$. Thus,
\begin{eqnarray*}
\beta^{E_{0}}(h_S)&=&\lim_{l\rightarrow\infty}\frac{|h_S(p_0^{01^l})-h_S(p_{01})|}{|p_0^{01^l}-p_{01}|}\\
&=&\frac{|h_S(p_0)-h_S(p_{01})|}{|p_0-p_{01}|}\lim_{l\rightarrow\infty} 2^l\big|\frac{r^{0}(h_S)}{1+r^{0}(h_S)}\cdot\frac{r^{01}(h_S)}{1+r^{01}(h_S)}\cdot\cdots\cdot\frac{r^{01^{l-1}}(h_S)}{1+r^{01^{l-1}}(h_S)}\big|\\&=&0,\\
\alpha^{E_{1}}(h_S)&=&\lim_{l\rightarrow\infty}\frac{|h_S(p_1^{10^l})-h_S(p_{01})|}{|p_1^{10^l}-p_{01}|}\\
&=&\frac{|h_S(p_1)-h_S(p_{01})|}{|p_1-p_{01}|}\lim_{l\rightarrow\infty} 2^l\big|\frac{1}{1+r^{1}(h_S)}\cdot\frac{1}{1+r^{10}(h_S)}\cdot\cdots\cdot\frac{1}{1+r^{10^{l-1}}(h_S)}\big|\\&=&0.
\end{eqnarray*}
 And for $h_A$, similar argument yields that,
\begin{eqnarray*}
\beta^{E_{0}}(h_A)&=&\lim_{l\rightarrow\infty}\frac{|h_A(p_0^{01^l})-h_A(p_{01})|}{|p_0^{01^l}-p_{01}|}=\infty,\\
\alpha^{E_{1}}(h_A)&=&\lim_{l\rightarrow\infty}\frac{|h_A(p_1^{10^l})-h_A(p_{01})|}{|p_1^{10^l}-p_{01}|}=\infty.
\end{eqnarray*}

On the other hand, by symmetry, $h_S(p_0^{01^l})-h_S(p_{01})$ and $h_S(p_1^{10^l})-h_S(p_{01})$ have  same signs, and
$h_A(p_0^{01^l})-h_A(p_{01})$ and $h_A(p_1^{10^l})-h_A(p_{01})$ have opposite signs for $l$ large enough.

Since $h=h_A+h_S$ on $E$, we then obtain that $\beta^{E_{0}}(h)=\infty$, $\alpha^{E_{0}}(h)=\infty$, and $h(p_0^{01^l})-h(p_{01})$ and $h(p_1^{10^l})-h(p_{01})$ have opposite signs for $l$ large enough. It follows that the restriction of $h$ to $E$ can not attain an extremum at $p_{01}$.
\end{proof}

\begin{proof}[Proof of Theorem \ref{thm1}]
From the definition of $r^\tau$ and the continuity of $h|_{E}$ we know that if $r^\tau\in(0,+\infty)$ for any word $\tau$, then $h|_E$ is strictly monotone. Conversely, if $r^\tau\in(-\infty,0)$ for some $\tau$, then $h|_E$ has at least one extremum. 

It is easy to check that if $r^\tau\in(\frac{1}{4},4)$ for some $\tau$, then $r^{\tau 0}\in(\frac{1}{4},4)$ and $r^{\tau1}\in(\frac{1}{4},4)$. So if we assume that $r\in(\frac{1}{4},4)$, then by iteration, we always have $r^\tau\in(\frac{1}{4},4)$ for any word $\tau$, and hence $h|_E$ is strictly monotone. Moreover, similar to the proof of Lemma \ref{lemma2}, we have $\lim_{l\rightarrow\infty}r^{ 0^l}=\frac{2}{3}$, $\lim_{l\rightarrow\infty}r^{1^l}=\frac{3}{2}$, and thus
\begin{eqnarray*}
\alpha&=&\lim_{l\rightarrow\infty}\frac{|h(p_1^{ 0^l})-h(p_0)|}{|p_1^{0^l}-p_0|}\\
&=&\frac{|h(p_1)-h(p_0)|}{|p_1-p_0|}\lim_{l\rightarrow\infty} 2^l\big|\frac{1}{1+r}\cdot\frac{1}{1+r^{0}}\cdot\cdots\cdot\frac{1}{1+r^{0^{l-1}}}\big|=\infty,
\end{eqnarray*}
\begin{eqnarray*}
\beta&=&\lim_{l\rightarrow\infty}\frac{|h(p_0^{ 1^l})-h(p_1)|}{|p_0^{1^l}-p_1|}\\
&=&\frac{|h(p_1)-h(p_0)|}{|p_1-p_0|}\lim_{l\rightarrow\infty} 2^l\big|\frac{r}{1+r}\cdot\frac{r^1}{1+r^{1}}\cdot\cdots\cdot\frac{r^{1^{l-1}}}{1+r^{1^{l-1}}}\big|=\infty,
\end{eqnarray*}
where $\alpha=\alpha^E(h)$, $\beta=\beta^E(h)$ for short.

If we assume $r=4$, then $r^0=4$ and $r^1\in(\frac{1}{4},4)$ by Lemma \ref{lemma1}. By iteration, we always have $r^\tau\in(\frac{1}{4},4]$ for any $\tau$, so $h|_{E}$ is still strictly monotone. However, since $r^{0^l}=4$ for any $l\geq 0$, we have $\alpha=\lim_{l\rightarrow\infty}\frac{|h(p_1^{ 0^l})-h(p_0)|}{|p_1^{0^l}-p_0|}=0$ instead. And there is a similar argument if we assume $r=\frac{1}{4}$. Thus (1) of the theorem follows. 

Now we consider the case that $r\in(-\infty,\frac{1}{4})\cup(4,+\infty)\cup\{\infty\}$. If in particular $r=-1$, then $h|_E$ has at least one extremum. But in this case, $r^0=\frac{1}{4}$ and $r^1=4$, so both $h|_{E_0}$ and $h|_{E_1}$ are strictly monotone as we discussed before. Thus $h|_E$ has a single extremum at $p_{01}$.  If in addition $r\neq -1$, then by Lemma \ref{lemma2}, $p_{01}$ cannot be an extreme point, and we can easily check from Fig. 3 that either $r^{0}\in(\frac{1}{4}, 4)$ and $r^1\notin [\frac{1}{4},4]$ when $r\in(-1,\frac{1}{4})$, or $r^0\notin[\frac{1}{4},4]$ and $r^1\in(\frac{1}{4},4)$ when $r\in(-\infty,-1)\cup(4,+\infty)\cup\{\infty\}$.  So that either $h|_{E_0}$ is strictly monotone and $r^1\notin[\frac{1}{4},4]$, or $h|_{E_1}$ is strictly monotone and $r^0\notin[\frac{1}{4},4]$. Thus we need continue to consider the local behavior of $h$ restricted to $E_1$ or $E_0$. Repeat this process recursively, by Lemma \ref{lemma2}, using a binary search, one can find that there is no more than one extremum for $h|_E$. 

Now we show that there does exist an extremum in this case. If $r\in(\frac{17}{3},+\infty)\cup\{\infty\}$, then $r^0<0$, so $h$ attains an extremum in $E_0$. If $r=\frac{17}{3}$, then $r^0=\infty$ and $r^{00}=-1$ and thus $h$ attains an extremum at $p_{01}^{00}$. If $r\in(4,\frac{17}{3})$, then since $x=4$ is the single unstable fixed point of the mapping $f(x)=\frac{8+3x}{17-3x}$, there exists an number $k\in\mathbb{N}$ such that $r^{0^k}\in[\frac{17}{3},+\infty)\cup\{\infty\}$. Then the case is  similar to the case that $r\in[\frac{17}{3},+\infty)\cup\{\infty\}$. In summary, we have proved that $h$ attains at least one extremum in $E$ when $r\in(4,+\infty)\cup\{\infty\}$. For the case $r\in[0,\frac{1}{4})$, it is similar by symmetric consideration. For the case $r\in(-\infty,0)$, the existence of extrema is obvious. Thus combining the binary search as we described before, we have proved that there indeed exists only one extremum for $h|_E$ when $r\notin[\frac{1}{4},4]$, and thus (2) of the theorem holds.  The boundary behavior is similar to the case for $r\in(\frac{1}{4},4)$.

Thus we have complete the proof of the theorem. 
\end{proof}

\begin{remark}
In case of $r\notin[\frac{1}{4},4]$, the above proof actually yields a binary searching algorithm for the unique extreme point of $h|_E$. In fact, we can recursively define a parameter $\theta(r)$ as a function of $r$ as
 \begin{equation*}
   \theta(r)=
      \begin{cases}
     \frac{1}{2},\quad &\text{if } r=-1,\\
  \frac{1}{2}+\frac{1}{2}\theta(r^1), \quad &\text{if } r\in(-1,\frac{1}{4}),\\
      \frac{1}{2}\theta(r^0),  &\text{otherwise}.\\
      \end{cases}
      \end{equation*}
      Then $p_0+\theta(r)(p_1-p_0)$ is the unique extreme point of $h|_E$.
\end{remark}

\section{restrictions of eigenfunctions to edges}

Now we come to the restrictions of Laplacian eigenfunctions to edges. Let $u$ be a Dirichlet or Neumann eigenfunction associated with an eigenvalue $\lambda$, i.e., $-\Delta u=\lambda u$. We use $m_0$ and $m_1$ to denote its level of birth and level of fixation respectively. For each $m\geq m_0$, we write $\lambda_m$ the graph eigenvalue such that $-\Delta_m u=\lambda_m u$ holds. Let $E$ be a $m$-edge in $\mathcal{SG}$. Throughout this section, we always assume that $u$ is non-constant on $E$. We still use the notations $p_0$, $p_1$, $p_{01}$, $p_0^\tau$, $p_1^\tau$, $p_{01}^\tau$, $r$, $r^\tau$ as we did for the harmonic functions, where $\tau$ is any word $\tau=\tau_1\cdots\tau_l$ with $\tau_i=0$ or $1$ and $|\tau|=l\geq 0$. In particular, 
\begin{equation*}
r=\frac{u(p_1)-u(p_{01})}{u(p_{01})-u(p_0)}\text{ and } r^\tau=\frac{u(p_1^\tau)-u(p_{01}^\tau)}{u(p_{01}^\tau)-u(p_0^\tau)}.
\end{equation*}
By using the extension algorithm (\ref{eq4}), we could write 
\begin{equation}\label{r}
r=\frac{(6-6\lambda_{m+1}+\lambda_{m+1}^2)u(p_1)-(4-\lambda_{m+1})u(p_0)-2u(p_2)}{(4-\lambda_{m+1})u(p_1)+2u(p_2)-(6-6\lambda_{m+1}+\lambda_{m+1}^2)u(p_0)},
\end{equation}
where $p_2$ is the other vertex in the $m$-cell containing $E$.
There is a similar formula for $r^\tau$ with $m+1$ replaced by $m+l+1$ and $p_i$ replaced by $p_i^\tau$.

Analogous to (\ref{eq7}), there is also an algorithm which enables us to calculate the restriction of a Laplaician eigenfunctions to any edge in $\mathcal{SG}$. Let $p_{001}^\tau$ be the midpoint of the edge connecting $p_0^\tau$ and $p_{01}^\tau$. Then
\begin{eqnarray}\label{eq32}
u(p_{001}^\tau)&=&\frac{4-\lambda_{m+l+2}}{5-\lambda_{m+l+2}}u(p_{01}^\tau)\nonumber\\
&&+\big(\frac{3-\lambda_{m+l+2}}{(2-\lambda_{m+l+2})(5-\lambda_{m+l+2})}+\frac{1}{(2-\lambda_{m+l+2})(5-\lambda_{m+l+2})(5-\lambda_{m+l+1})}\big)u(p_{0}^\tau)\\\nonumber
&&-\big(\frac{1}{(2-\lambda_{m+l+2})(5-\lambda_{m+l+2})}+\frac{1}{(2-\lambda_{m+l+2})(5-\lambda_{m+l+2})(5-\lambda_{m+l+1})}\big)u(p_1^\tau).
\end{eqnarray}
See Algorithm 2.5 in \cite{DSV} for a proof of (\ref{eq32}) by using the extension rule (\ref{eq4}).

We will prove Theorem \ref{thm2} and Theorem \ref{thm3} in this section.

We need some lemmas.

\begin{lemma}\label{lemma31}
For an integer $l\geq 0$, we write $a_l=(2-\lambda_{m+l+2})(4-\lambda_{m+l+2})$, $b_l=3-\lambda_{m+l+2}$ and $c_l=(2-\lambda_{m+l+2})^2(5-\lambda_{m+l+2})-3$. Then
for any word $\tau$ with $|\tau|=l$, we have
\begin{equation}\label{eq31}
r^{\tau 0}=\frac{a_l+b_l r^\tau}{c_l-b_l r^\tau}\text{ and } r^{\tau 1}=\frac{c_lr^\tau-b_l}{a_l r^\tau+b_l}. 
\end{equation}
Furthermore, the two mappings in (\ref{eq31}) are both bijections: $\mathbb{R}\cup \{\infty\}\rightarrow \mathbb{R}\cup \{\infty\}$. In particular, $r^{\tau 0}=-1$, $r^{\tau 1}=\frac{c_l}{a_l}$ when $r^\tau=\infty$.
\end{lemma}

\begin{proof}
The proof is  similar to that of Lemma \ref{lemma1}. Noticing that $m\geq m_0$, we have $\lambda_{m+k}\in(0,2)\cup(3,5)\setminus\{4\}$ for any $k\geq 2$, since $\lambda_{m_0}=2,5$ or  $6$. Hence we always have $a_l,b_l\neq 0$ for any $l\geq 0$. By (\ref{eq32}) and the definition of $r^\tau$ and $r^{\tau0}$, we could get the first equality in (\ref{eq31}). The second equality in (\ref{eq31}) follows by a symmetric consideration. It can be directly checked that the two mappings are bijections. 
\end{proof}

\begin{lemma}\label{lemma32}
For any word $\tau$ with $|\tau|=l\geq 0$, we have

(1) if $r^\tau=4-\lambda_{m+l+1}$, then $r^{\tau 0}=4-\lambda_{m+l+2}$;

(2) if $r^\tau=\frac{1}{4-\lambda_{m+l+1}}$, then $r^{\tau 1}=\frac{1}{4-\lambda_{m+l+2}}$;

(3) if $r^\tau=-1$, then $r^{\tau 0}=\frac{1}{4-\lambda_{m+l+2}}$ and $r^{\tau 1}=4-\lambda_{m+l+2}$. 
\end{lemma}

\begin{proof}
Notice that from $(\ref{eq2})$, $\lambda_{m+l+1}=\lambda_{m+l+2}(5-\lambda_{m+l+2})$. So if $r^\tau=4-\lambda_{m+l+1}$, then by Lemma \ref{lemma31},
\begin{eqnarray*}
r^{\tau 0}&=&\frac{(2-\lambda_{m+l+2})(4-\lambda_{m+l+2})+(3-\lambda_{m+l+2})(4-\lambda_{m+l+1})}{(2-\lambda_{m+l+2})^2(5-\lambda_{m+l+2})-3-(3-\lambda_{m+l+2})(4-\lambda_{m+l+1})}\\
&=&\frac{20-25\lambda_{m+l+2}+9\lambda_{m+l+2}^2-\lambda_{m+l+2}^3}{5-5\lambda_{m+l+2}+\lambda_{m+l+2}^2}\\
&=&4-\lambda_{m+l+2}.
\end{eqnarray*}

On the other hand, since $r^{\tau1}=\big(\frac{a_l+b_l (r^\tau)^{-1}}{c_l-b_l (r^\tau)^{-1}}\big)^{-1}$, we have $r^{\tau 1}=\frac{1}{4-\lambda_{m+l+2}}$ when $r^\tau=\frac{1}{4-\lambda_{m+l+1}}$.

If $r^\tau=-1$, then
\begin{eqnarray*}
r^{\tau 0}&=&\frac{(2-\lambda_{m+l+2})(4-\lambda_{m+l+2})-(3-\lambda_{m+l+2})}{(2-\lambda_{m+l+2})^2(5-\lambda_{m+l+2})-3+(3-\lambda_{m+l+2})}\\
&=&\frac{5-5\lambda_{m+l+2}+\lambda_{m+l+2}^2}{20-25\lambda_{m+l+2}+9\lambda_{m+l+2}^2-\lambda_{m+l+2}^3}\\
&=&\frac{1}{4-\lambda_{m+l+2}},
\end{eqnarray*}
and similarly, $r^{\tau 1}=4-\lambda_{m+l+2}$.
\end{proof}

Similar to Lemma \ref{lemma2}, we have
\begin{lemma}\label{lemma33}
For any word $\tau$, the restriction of $u$ to $E$ can not attain an extremum at $p_{01}^\tau$ if $r^\tau\neq -1$.
\end{lemma}

\begin{proof}
As in Lemma \ref{lemma2}, we only need to prove that $u$ can not attain an extremum at $p_{01}$ if $r\neq -1$. Suppose $r\neq -1$, then we divide the restriction of $u$ to $E$ into its symmetric part $u_S$ and antisymmetric part $u_A$. Then $r(u_S)=-1$ and $r(u_A)=1$.

 Since $r(u_S)=-1$, by Lemma \ref{lemma32}, we have $r^{1}(u_S)=4-\lambda_{m+2}$ and $r^{0}(u_S)=\frac{1}{4-\lambda_{m+2}}$, and then $r^{10^l}(u_S)=4-\lambda_{m+l+2}$ and $r^{01^l}(u_S)=\frac{1}{4-\lambda_{m+l+2}}$ for any $l\geq 0$. This gives that $\lim_{l\rightarrow\infty}r^{ 10^l}(u_S)=4$ and $\lim_{l\rightarrow\infty}r^{01^l}(u_S)=\frac{1}{4}$. 
 
 As for $u_A$, since $r(u_A)=1$, we have $r^1(u_A)\neq 4-\lambda_{m+2}$, then since the mapping $f_l(x)=\frac{a_{l+1}+b_{l+1}x}{c_{l+1}-b_{l+1}x}$ is bijective for any $l\geq 0$, we have $r^{10^l}(u_A)\neq 4-\lambda_{m+l+2}$ for any $l\geq 0$. Moreover, noticing that $|\frac{df_l}{dx}|_{4-\lambda_{m+l+2}}|>1$ for any $l\geq 0$, we could not have $\lim_{l\rightarrow\infty}r^{10^l}(u_A)=4$. On the other hand, since $\lim_{l\rightarrow\infty}a_l=8$, $\lim_{l\rightarrow\infty}b_l=3$, $\lim_{l\rightarrow\infty}c_l=17$, and $x=\frac{2}{3}$ is the unique stable fixed point of $f(x)=\frac{8+3x}{17-3x}$, we have $\lim_{l\rightarrow\infty}r^{10^l}(u_A)=\frac{2}{3}$. A similar argument yields that $\lim_{l\rightarrow\infty}r^{01^l}(u_A)=\frac{3}{2}$.

The remaining proof is  same as that of Lemma \ref{lemma2}.
\end{proof}

First we consider the case $m\geq m_1-1$. In this case, $\varepsilon_{m+l+1}=-1$ for any $l\geq 0$.

\begin{proof}[Proof of Theorem \ref{thm2}]
Since $m\geq m_1-1$, for any word $\tau$ with $|\tau|=l\geq 0$, we have $\lambda_{m+l+1}\leq \frac{5-\sqrt{5}}{2}$ and $\lambda_{m+l+2}\leq\frac{5-\sqrt{15+2\sqrt{5}}}{2}$. So it is easy to check that
\begin{equation*}
-1<-\frac{b_l}{a_l}<0<\frac{1}{4-\lambda_{m+l+1}}<4-\lambda_{m+l+1}<\frac{c_l}{b_l}
\end{equation*}
and
$a_l,b_l,c_l>0$. Combining this with Lemma \ref{lemma31} and Lemma \ref{lemma32}, we get that the graphs of $r^{\tau 0}$ and $r^{\tau 1}$ in terms of $r^\tau$ look like what Fig. 4 presents.

\begin{figure}[h]
\begin{center}
\includegraphics[width=11cm]{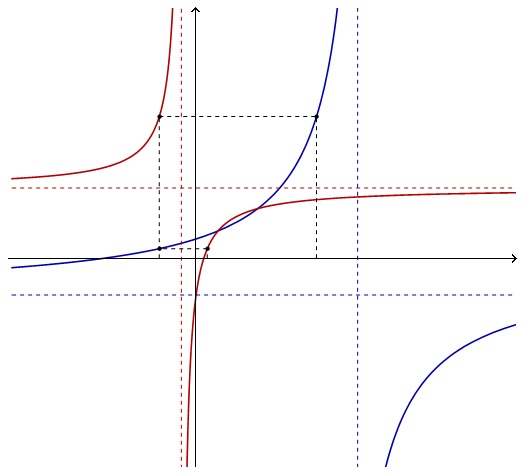}
\begin{center}
\begin{picture}(0,0) \thicklines
\put(154,146){$r^\tau$}
\put(124,176){$r^{\tau1}$}
\put(134,92){$r^{\tau0}$}
\put(28,280){$r^{\tau0}$}
\put(-71,280){$r^{\tau1}$}
\put(-70,141){\tiny$-1$}\put(-59,140){\tiny$-\frac{b_l}{a_l}$}\put(-38,140){\tiny$\frac{1}{4-\lambda_{m+l+1}}$}\put(10,140){\tiny$4-\lambda_{m+l+1}$}\put(59,140){\tiny$\frac{c_l}{b_l}$}
\put(-40,197){\tiny$\frac{c_l}{a_l}$}\put(-40,118){\tiny$-1$}\put(-40,236){\tiny$4-\lambda_{m+l+2}$}\put(-100,160){\tiny$\frac{1}{4-\lambda_{m+l+2}}$}
\end{picture}
\textbf{\\Figure 4. Graphs of $r^{\tau0}$ and $r^{\tau1}$ of $r^\tau$}
\end{center}\end{center}
\end{figure}

The proof is  similar to that of Theorem \ref{thm1}. If $r\in[\frac{1}{4-\lambda_{m+1}}, 4-\lambda_{m+1}]$, by iteration, we will have $ r^\tau\in[\frac{1}{4-\lambda_{m+l+1}}, 4-\lambda_{m+l+1}]$ for any word $\tau$ with $|\tau|=l\geq 0$. Thus in this case $u$ is strictly monotone along the line segment $E$. Otherwise if $r\notin[\frac{1}{4-\lambda_{m+1}}, 4-\lambda_{m+1}]$, then there must exist some $\tau$ such that $r^\tau<0$ and there is a binary searching algorithm so that $u$ has exactly one extremum along $E$.

As for the boundary behavior of $u$ along $E$, we need to consider $\lim_{l\rightarrow\infty}r^{0^l}$ and $\lim_{l\rightarrow\infty}r^{1^l}$. Similar to the proof of Lemma \ref{lemma33}, if $r=4-\lambda_{m+1}$, then we have
\begin{equation*}
\lim_{l\rightarrow\infty}r^{0^l}=\lim_{l\rightarrow\infty}(4-\lambda_{m+l+1})=4,
\end{equation*} 
which leads to 
\begin{eqnarray*}
\alpha&=&\lim_{l\rightarrow\infty}\frac{|u(p_1^{0^l})-u(p_0)|}{|p_1^{0^l}-p_0|}\\
&=&\frac{|u(p_1)-u(p_0)|}{|p_1-p_0|}\lim_{l\rightarrow\infty}2^l\frac{1}{1+r}\cdot\frac{1}{1+r^0}\cdot\cdots\cdot\frac{1}{1+r^{0^{l-1}}}=0,
\end{eqnarray*}
 where $\alpha=\alpha^E(u)$ for short. Otherwise if $r\neq 4-\lambda_{m+1}$, then  $\lim_{l\rightarrow\infty}r^{0^l}=\frac{2}{3}$, which leads to that $\alpha=\lim_{l\rightarrow\infty}\frac{|u(p_1^{0^l})-u(p_0)|}{|p_1^{0^l}-p_0|}=\infty$. There is a similar argument for $\beta=\beta^E(u)$ by considering $\lim_{l\rightarrow\infty}r^{1^l}$ instead.
 
 Thus we have proved the theorem.
 \end{proof}

Similar to the case for harmonic functions, we have a binary searching algorithm for the unique extreme point of $u|_E$ in case of $r\notin[\frac{1}{4-\lambda_{m+1}}, 4-\lambda_{m+1}]$.

\begin{remark}
When $r\notin[\frac{1}{4-\lambda_{m+1}}, 4-\lambda_{m+1}]$, then $p_0+\theta(r)p_1$ is the unique extreme point of $u|_E$, where $\theta(r)$ is a function of $r$ defined recursively as
 \begin{equation*}
   \theta(r^\tau)=
      \begin{cases}
     \frac{1}{2},\quad &\text{if } r^\tau=-1,\\
  \frac{1}{2}+\frac{1}{2}\theta(r^{\tau1}), \quad &\text{if } r^\tau\in(-1,\frac{1}{4-\lambda_{m+|\tau|+1}}),\\
      \frac{1}{2}\theta(r^{\tau0}),  &\text{otherwise},\\
      \end{cases}
      \end{equation*}
      for any word $\tau$.
\end{remark}

\begin{corollary}\label{cor31}
Assume $m\geq m_1-1$. Let $\tau$ be a word with $|\tau|=l\geq 0$. Then

(1) $p_{01}^\tau$ is an extreme point of $u$ along $E$ if and only if $r^\tau=-1$;

(2) if $\tau\neq 0^l$ (write $\tau=\tau'10^k$ with $k\geq 0$), then $p_0^\tau=p_{01}^{\tau'}=p_1^{\tau'01^k}$ is an extreme point of $u$ along $E$ if and only if $r^\tau=4-\lambda_{m+l+1}$,  and if so, $r^{\tau'}=-1$ and $r^{\tau'01^k}=\frac{1}{4-\lambda_{m+l+1}}$;

(3) if $\tau\neq 1^l$ (write $\tau=\tau'01^k$ with $k\geq 0$), then $p_1^\tau=p_{01}^{\tau'}=p_0^{\tau'10^k}$ is an extreme point of $u$ along $E$ if and only if $r^\tau=\frac{1}{4-\lambda_{m+l+1}}$,  and if so, $r^{\tau'}=-1$ and $r^{\tau'10^k}=4-\lambda_{m+l+1}$.
\end{corollary}

\begin{proof}
For (1), if $r^\tau=-1$, then from Lemma \ref{lemma32}, we have 
$r^{\tau0}=\frac{1}{4-\lambda_{m+l+2}}$
 and 
$r^{\tau1}={4-\lambda_{m+l+2}}$. Thus from Theorem \ref{thm2}, both $u|_{E_{\tau 0}}$ and $u|_{E_{\tau1}}$ are monotone. This gives that $p_{01}^\tau$ is an extreme point of $u$ since $r^\tau=-1$. Combing this with Lemma \ref{lemma33}, we have proved (1).

For (2), by Lemma \ref{lemma31} and Lemma \ref{lemma32}, it is easy to check that 
\begin{equation*}
r^\tau=4-\lambda_{m+l+1}\Longleftrightarrow r^{\tau'}=-1\Longleftrightarrow r^{\tau'01^k}=\frac{1}{4-\lambda_{m+l+1}}.\end{equation*}
Then (2) follows from (1).

The proof of (3) is similar to that of (2).
\end{proof}

Now we turn to the general case $m\geq m_0$. In particular, we mainly focus on the case that $m_0\leq m<m_1-1$ (so we naturally assume $m_0<m_1-1$). Write $\{\varepsilon_{m_0+i}\}_{i\geq 1}$ the associated $\varepsilon$-sequence of $u$ and $\lambda$. It is easy to check that  $\varepsilon_{m_1-1}=1$. From now on, we always denote $\tilde{u}$ a new eigenfunction such that $\tilde{u}|_{V_m}=u|_{V_m}$ whose associated $\varepsilon$-sequence $\{\tilde{\varepsilon}_{m_0+i}\}_{i\geq 1}$ satisfies
  \begin{equation*}
  \tilde{\varepsilon}_{m_0+i}=
      \begin{cases}
     \varepsilon_{m_0+i},\quad &\text{if } i\neq m_1-m_0-1,\\
  -1, \quad &\text{if } i=m_1-m_0-1.\\
      \end{cases}
      \end{equation*}
      We write $\{\tilde{\lambda}_{m_0+i}\}_{i\geq 0}$ and $\tilde{\lambda}$ the associated graph eigenvalues and eigenvalue of $\tilde{u}$. Of cause, $\lambda_{m_0+i}=\tilde{\lambda}_{m_0+i}$ for $i\leq m_1-m_0-2$.
     We additional require $m>m_0$ if $\lambda_{m_0}=6$ to avoid the possibility that $m_1-1=m_0+1$ since we could not allow $\tilde{\varepsilon}_{m_0+1}=-1$ and $\tilde{\lambda}_{m_0+1}=2$.
     
     We write $\tilde{r}$, $\tilde{r}^\tau$, $\tilde{a}_l$, $\tilde{b}_l$, $\tilde{c}_l$ of $\tilde{u}$ on $E$, analogous to those of $u$ on $E$. Then we have the following lemma.
     
\begin{lemma}\label{lemma34}
Assume $m_0\leq m<m_1-1$ ($m_0<m<m_1-1$ if $\lambda_{m_0}=6$). Then for any $(m_1-1)$-edge $E_\tau$ contained in $E$, we have

(1) if $\tilde{r}^\tau\in(\frac{1}{4-\tilde{\lambda}_{m_1}}, 4-\tilde{\lambda}_{m_1})$, then $r^\tau\notin[\frac{1}{4-\lambda_{m_1}}, 4-\lambda_{m_1}]$;

(2) if $\tilde{r}^\tau=\frac{1}{4-\tilde{\lambda}_{m_1}}$, then $r^\tau=\frac{1}{4-\lambda_{m_1}}$; if $\tilde{r}^\tau=4-\tilde{\lambda}_{m_1}$, then $r^\tau=4-\lambda_{m_1}$;

(3) if $\tilde{r}^\tau\notin[\frac{1}{4-\tilde{\lambda}_{m_1}}, 4-\tilde{\lambda}_{m_1}]$, then $r^\tau\in(\frac{1}{4-\lambda_{m_1}}, 4-\lambda_{m_1})$.
\end{lemma}

\begin{proof}
Without loss of generality, we assume $m=m_1-2$, i.e., $E$ is a $(m_1-2)$-edge, otherwise we just need to consider every $(m_1-2)$-edge contained in $E$ instead of $E$ separately. Then $|\tau|=1$, i.e., $\tau=0$ or $1$. From the definition of $\tilde{u}$, since $E$ is a $(m_1-2)$-edge, we have $\tilde{u}(p_i)=u(p_i)$, $i=0,1,2$, and $\tilde{\lambda}_{m_0+i}=\lambda_{m_0+i}$ for $0\leq i\leq m_1-m_0-2$. 

First we assume that  $m_0<m(=m_1-2)$.  We make a virtual $(m-1)$-cell such that $E$ is a half of one of its edge, as shown in Fig  5. Let $p_0$, $q_1$, $q_2$ be its boundary vertices, $p_1$ be the midpoint of the $(m-1)$-edge joining $p_0$ and $q_1$. 

\begin{figure}[h]
\begin{center}
\includegraphics[width=6.5cm]{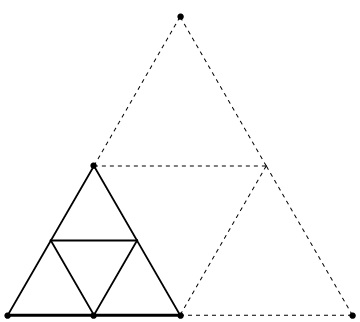}
\begin{center}
\begin{picture}(0,0) \thicklines
\put(-96,14){$p_0$}\put(-3,14){$p_1$}\put(-49,14){$p_{01}$}
\put(90,14){$q_1$}\put(-3,182){$q_2$}\put(-52,104){$p_2$}\put(-53,2){$E$}
\end{picture}
\textbf{\\Figure 5. A virtual $(m-1)$-cell containing $E$}
\end{center}\end{center}
\end{figure}

Comparing to (\ref{eq4}), we choose the values of $u$ at the virtual vertices $q_i$ to be
\begin{equation*}
u(q_i)=-\frac{4-\lambda_{m}}{6-\lambda_{m}}u(p_0)+\frac{(4-\lambda_{m})(5-\lambda_{m})}{6-\lambda_{m}}u(p_i)-\frac{2(5-\lambda_{m})}{6-\lambda_{m}}u(p_j),
\end{equation*}
for distinct $i,j\in\{1,2\}$,
in order the extension relation
\begin{equation*}
u(p_i)=\frac{(4-\lambda_{m})\big(u(p_0)+u(q_i)\big)+2u(q_j)}{(2-\lambda_{m})(5-\lambda_{m})}
\end{equation*}
to hold. Then if we define $r^{*}=\frac{u(q_1)-u(p_1)}{u(p_1)-u(p_0)}$, by Lemma \ref{lemma31}, we have
\begin{equation*}
r=\frac{a_{-1}+b_{-1}r^{*}}{c_{-1}-b_{-1}r^{*}} \text{ and }\tilde{r}=\frac{\tilde{a}_{-1}+\tilde{b}_{-1}r^{*}}{\tilde{c}_{-1}-\tilde{b}_{-1}r^{*}} 
\end{equation*} where $a_{-1}, b_{-1}, c_{-1}$ are functions of $\lambda_{m+1}$, and $\tilde{a}_{-1}, \tilde{b}_{-1}, \tilde{c}_{-1}$ are functions of $\tilde{\lambda}_{m+1}$, which are same as those in Lemma \ref{lemma31} with $l=-1$ and $m=m_1-2$.

Analogously, $r^0=\frac{a_0+b_0r}{c_0-b_0r}$, $\tilde{r}^0=\frac{\tilde{a}_0+\tilde{b}_0\tilde{r}}{\tilde{c}_0-\tilde{b}_0\tilde{r}}$, $r^1=\frac{c_0r-b_0}{a_0r+b_0}$, $\tilde{r}^1=\frac{\tilde{c}_0\tilde{r}-\tilde{b}_0}{\tilde{a}_0\tilde{r}+\tilde{b}_0}$, where $a_0,b_0,c_0,\tilde{a}_0,\tilde{b}_0,\tilde{c}_0$ are functions of $\lambda_{m+2}$ or $\tilde{\lambda}_{m+2}$.

Since $m=m_1-2>m_0$, it is easy to check that $\lambda_{m}\in(0,\frac{5-\sqrt{5}}{2}]\cup[\frac{5+\sqrt{5}}{2},5)\cup\{3\}\setminus\{1,4\}$, and then $\lambda_{m+1}=\frac{5+\sqrt{25-4\lambda_m}}{2}\in (\frac{5+\sqrt{5}}{2},5)\setminus\{1,4\}$, $\tilde{\lambda}_{m+1}=\frac{5-\sqrt{25-4\lambda_m}}{2}\in (0,\frac{5-\sqrt{5}}{2})\setminus\{1,4\}$, and similarly $\lambda_{m+2},\tilde{\lambda}_{m+2}\in(0,\frac{5-\sqrt{5}}{2})\setminus\{1,4\}$.

Then we have $\tilde{b}_{-1}>0$, $-1<4-\lambda_{m}<\frac{\tilde{c}_{-1}}{\tilde{b}_{-1}}$, so the graph of  $\tilde{r}=\frac{\tilde{a}_{-1}+\tilde{b}_{-1}r^*}{\tilde{c}_{-1}-\tilde{b}_{-1}r^*}$ looks like what Fig. 6-1 presents. Similarly, we have $b_{-1}<0$, $\frac{c_{-1}}{b_{-1}}<-1<4-\lambda_m$ when $\lambda_{m+1}<4$ and $-1<\frac{c_{-1}}{b_{-1}}<4-\lambda_{m}$ when $\lambda_{m+1}>4$, see Fig. 6-2, 6-3 for the two possible cases of graphs of $r=\frac{a_{-1}+b_{-1}r^*}{c_{-1}-b_{-1}r^*}$.

For $\tilde{r}^0$, we have $\tilde{b}_0>0$ and $-1<4-\tilde{\lambda}_{m+1}<\frac{\tilde{c}_0}{\tilde{b}_0}$, see Fig. 6-4 for the graph of $\tilde{r}^0=\frac{\tilde{a}_0+\tilde{b}_0\tilde{r}}{\tilde{c}_0-\tilde{b}_0\tilde{r}}$. Similarly for $r^0$, we have $b_0>0$ and $-1<4-\lambda_{m+1}<\frac{c_0}{b_0}$, see Fig. 6-5 for the graph of $r^0=\frac{a_0+b_0r}{c_0-b_0r}$.

For $\tilde{r}^1$, we have $\tilde{b}_0>0$ and $-1<-\frac{\tilde{b}_0}{\tilde{a}_0}<\frac{1}{4-\tilde{\lambda}_{m+1}}$, see Fig. 6-6 for the graph of $\tilde{r}^1=\frac{\tilde{c}_0\tilde{r}-\tilde{b}_0}{\tilde{a}_0\tilde{r}+\tilde{b}_0}$. Similarly for $r^1$, we have $b_0>0$ and $-1<-\frac{a_0}{b_0}<\frac{1}{4-\lambda_{m+1}}$ when $\lambda_{m+1}<4$, $\frac{1}{4-\lambda_{m+1}}<-1<-\frac{b_0}{a_0}$ when $\lambda_{m+1}>4$, see Fig. 6-7, 6-8 for the two possible graphs of $r^1=\frac{c_0r-b_0}{a_0r+b_0}$.

\begin{figure}[h]
\begin{center}
\includegraphics[width=4.5cm]{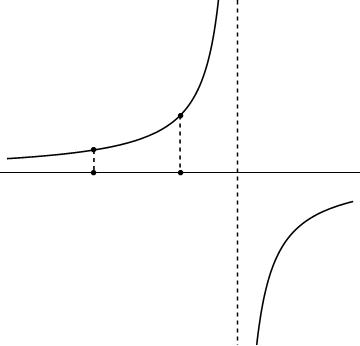}\hspace{2cm}
\includegraphics[width=4.5cm]{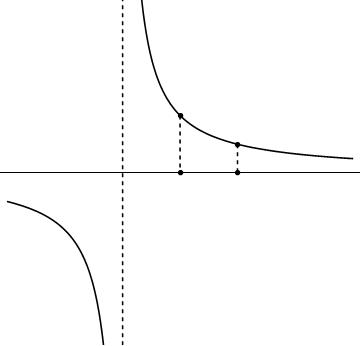}
\includegraphics[width=4.5cm]{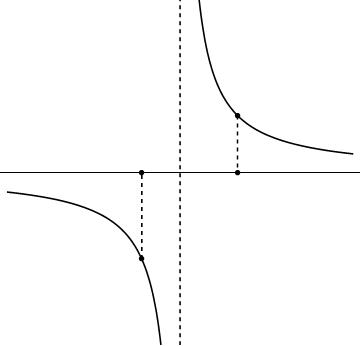}\hspace{2cm}
\includegraphics[width=4.5cm]{figure6a.jpg}
\includegraphics[width=4.5cm]{figure6a.jpg}\hspace{2cm}
\includegraphics[width=4.5cm]{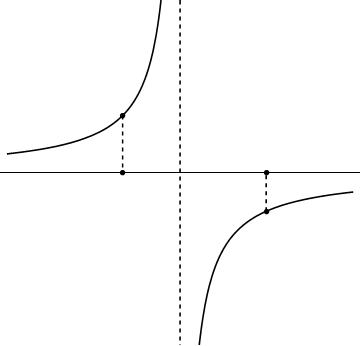}
\includegraphics[width=4.5cm]{figure6d.jpg}\hspace{2cm}
\includegraphics[width=4.5cm]{figure6a.jpg}
\begin{center}
\begin{picture}(0,0) \thicklines
\put(-175,451){\Tiny$(-1)$}
\put(-131,447){\Tiny$-1$}
\put(-141,470){\Tiny$\frac{1}{4-\tilde{\lambda}_{m+1}}$}
\put(-104,447){\Tiny$4-\lambda_{m}$}
\put(-109,477){\Tiny$4-\tilde{\lambda}_{m+1}$}
\put(-73,444){\Tiny$\frac{\tilde{c}_{-1}}{\tilde{b}_{-1}}$}
\put(-87,507){\Tiny$\tilde{r}$}
\put(-30,452){\Tiny$r^*$}
\put(-40,500){\Small 6-1}

\put(14,451){\Tiny$(-1)$}
\put(60,445){\Tiny$\frac{{c}_{-1}}{{b}_{-1}}$}
\put(89,447){\Tiny$-1$}
\put(80,480){\Tiny$\frac{1}{4-{\lambda}_{m+1}}$}
\put(106,447){\Tiny$4-\lambda_{m}$}
\put(107,467){\Tiny$4-{\lambda}_{m+1}$}
\put(83,507){\Tiny${r}$}
\put(160,452){\Tiny$r^*$}
\put(140,500){\Small 6-2}
\put(120,489){\Tiny$(\lambda_{m+1}<4)$}

\put(-175,327){\Tiny$(-1)$}
\put(-120,323){\Tiny$-1$}
\put(-126,293){\Tiny$\frac{1}{4-{\lambda}_{m+1}}$}
\put(-95,321){\Tiny$\frac{{c}_{-1}}{{b}_{-1}}$}
\put(-80,323){\Tiny$4-\lambda_{m}$}
\put(-84,353){\Tiny$4-{\lambda}_{m+1}$}
\put(-30,327){\Tiny$r^*$}
\put(-85,380){\Tiny$r$}
\put(-40,370){\Small 6-3}
\put(-60,359){\Tiny$(\lambda_{m+1}>4)$}

\put(14,327){\Tiny$(-1)$}
\put(58,323){\Tiny$-1$}
\put(48,346){\Tiny$\frac{1}{4-\tilde{\lambda}_{m+2}}$}
\put(81,322){\Tiny$4-\tilde{\lambda}_{m+1}$}
\put(80,353){\Tiny$4-\tilde{\lambda}_{m+2}$}
\put(116,321){\Tiny$\frac{\tilde{c}_{0}}{\tilde{b}_{0}}$}
\put(100,383){\Tiny$\tilde{r}^0$}
\put(159,328){\Tiny$\tilde{r}$}
\put(149,376){\Small 6-4}

\put(-175,203){\Tiny$(-1)$}
\put(-131,199){\Tiny$-1$}
\put(-141,221){\Tiny$\frac{1}{4-{\lambda}_{m+2}}$}
\put(-108,199){\Tiny$4-{\lambda}_{m+1}$}
\put(-109,229){\Tiny$4-{\lambda}_{m+2}$}
\put(-73,199){\Tiny$\frac{{c}_{0}}{{b}_{0}}$}
\put(-89,258){\Tiny${r}^0$}
\put(-30,204){\Tiny${r}$}
\put(-40,252){\Small 6-5}

\put(14,203){\Tiny$(\frac{\tilde{c}_0}{\tilde{a}_0})$}
\put(68,199){\Tiny$-1$}
\put(58,230){\Tiny${4-\tilde{\lambda}_{m+2}}$}
\put(110,213){\Tiny$\frac{1}{4-\tilde{\lambda}_{m+1}}$}
\put(115,184){\Tiny$\frac{1}{4-\tilde{\lambda}_{m+2}}$}
\put(95,196){\Tiny$-\frac{\tilde{b}_{0}}{\tilde{a}_{0}}$}
\put(79,257){\Tiny$\tilde{r}^1$}
\put(159,203){\Tiny$\tilde{r}$}
\put(149,252){\Small 6-6}

\put(-175,80){\Tiny$(\frac{{c}_0}{{a}_0})$}
\put(-121,75){\Tiny$-1$}
\put(-131,105){\Tiny${4-{\lambda}_{m+2}}$}
\put(-79,88){\Tiny$\frac{1}{4-{\lambda}_{m+1}}$}
\put(-71,63){\Tiny$\frac{1}{4-{\lambda}_{m+2}}$}
\put(-94,73){\Tiny$-\frac{{b}_{0}}{{a}_{0}}$}
\put(-110,134){\Tiny${r}^1$}
\put(-30,80){\Tiny${r}$}
\put(-40,128){\Small 6-7}
\put(-60,117){\Tiny$(\lambda_{m+1}<4)$}

\put(14,80){\Tiny$(\frac{{c}_0}{{a}_0})$}
\put(89,75){\Tiny$-1$}
\put(48,97){\Tiny$\frac{1}{4-{\lambda}_{m+2}}$}
\put(50,74){\Tiny$\frac{1}{4-{\lambda}_{m+1}}$}
\put(79,104){\Tiny$4-{\lambda}_{m+2}$}
\put(116,74){\Tiny$-\frac{{b}_{0}}{{a}_{0}}$}
\put(100,134){\Tiny${r}^1$}
\put(159,80){\Tiny${r}$}
\put(149,128){\Small 6-8}
\put(129,117){\Tiny$(\lambda_{m+1}>4)$}
\end{picture}
\textbf{\\Figure 6. Graphs of $r, \tilde{r}, r^0, \tilde{r}^0, r^1, \tilde{r}^1$}
\end{center}\end{center}
\end{figure}

Looking at the graphs in Fig. 6, we can find that: if $r^*\in(-\infty,4-\lambda_m)$, then $r\in (-\infty,-1)\cup(4-\lambda_{m+1},+\infty)\cup\{\infty\}$, $\tilde{r}\in (-1,4-\tilde{\lambda}_{m+1})$, and then $r^0\notin[\frac{1}{4-\lambda_{m+2}}, 4-\lambda_{m+2}]$, $\tilde{r}^0\in(\frac{1}{4-\tilde{\lambda}_{m+2}}, 4-\tilde{\lambda}_{m+2})$; if $r^*\in(4-\lambda_{m},+\infty)$, then $r\in(-1,4-\lambda_{m+1})$, $\tilde{r}\in (-\infty,-1)\cup(4-\tilde{\lambda}_{m+1},+\infty)\cup\{\infty\}$, and then $r^0\in(\frac{1}{4-\lambda_{m+2}},4-\lambda_{m+2})$, $\tilde{r}^0\notin[\frac{1}{4-\tilde{\lambda}_{m+2}}, 4-\tilde{\lambda}_{m+2}]$; if $r^*=4-\lambda_m$, then $r=4-\lambda_{m+1}$, $\tilde{r}=4-\tilde{\lambda}_{m+1}$, and then $r^0=4-\lambda_{m+2}$, $\tilde{r}^0=4-\tilde{\lambda}_{m+2}$; and if $r^*=\infty$, then $r=\tilde{r}=-1$, and then $r^0=\frac{1}{4-\lambda_{m+2}}$, $\tilde{r}^0=\frac{1}{4-\tilde{\lambda}_{m+2}}$. Thus we have verified that the conclusion of the lemma is true for $\tau=0$. 

As for $r^1$ and $\tilde{r}^1$, we have: if $r^*\in(-\infty,-1)$, then $r^1\in(\frac{1}{4-\lambda_{m+2}},4-\lambda_{m+2})$ and $\tilde{r}^1\notin[\frac{1}{4-\tilde{\lambda}_{m+2}},4-\tilde{\lambda}_{m+2}]$; if $r^*\in(-1,+\infty)$, then $r^1\notin[\frac{1}{4-\lambda_{m+2}},4-\lambda_{m+2}]$ and $\tilde{r}^1\in(\frac{1}{4-\tilde{\lambda}_{m+2}},4-\tilde{\lambda}_{m+2})$; if $r^*=-1$, then $r^1=\frac{1}{4-\lambda_{m+2}}$ and $\tilde{r}^1=\frac{1}{4-\tilde{\lambda}_{m+2}}$; and if $r^*=\infty$, then $r^1=4-\lambda_{m+2}$ and $\tilde{r}^1=4-\tilde{\lambda}_{m+2}$. So the conclusion of the lemma is true for $\tau=1$.

Thus we have proved the lemma under the assumption that $m_0<m$.

 It remain to consider the case that $m_0=m(=m_1-2)$. In this case, $E$ is a $m_0$-edge, $\lambda_{m_0}=2$ or $5$, and $\varepsilon_{m_0+1}=1$, $\tilde{\varepsilon}_{m_0+1}=-1$ since $m_0+1=m_1-1$. We now verify the conclusion according to the value of $\lambda_{m_0}$.

Case 1: $\lambda_{m_0}=2$. It is easy to check that $6-6\lambda_{m_0+1}+\lambda_{m_0+1}^2=4-\lambda_{m_0+1}$, then by (\ref{r}) we have
\begin{equation*}
r=\frac{(4-\lambda_{m_0+1})u(p_1)-(4-\lambda_{m_0+1})u(p_0)-2u(p_2)}{(4-\lambda_{m_0+1})u(p_1)-(4-\lambda_{m_0+1})u(p_0)+2u(p_2)}=\frac{4-\lambda_{m_0+1}-2\theta}{4-\lambda_{m_0+1}+2\theta}
\end{equation*}
where $\theta=\frac{u(p_2)}{u(p_1)-u(p_0)}$ \big(including $\infty$ when $u(p_1)=u(p_0)$\big). Similarly, $\tilde{r}=\frac{4-\tilde{\lambda}_{m_0+1}-2\theta}{4-\tilde{\lambda}_{m_0+1}+2\theta}$. 

Notice that if $\theta=-1$, then $r=\frac{6-\lambda_{m_0+1}}{2-\lambda_{m_0+1}}=4-\lambda_{m_0+1}$ and similarly $\tilde{r}=4-\tilde{\lambda}_{m_0+1}$; and if $\theta=1$, then $r=\frac{2-\lambda_{m_0+1}}{6-\lambda_{m_0+1}}=\frac{1}{4-\lambda_{m_0+1}}$ and similarly $\tilde{r}=\frac{1}{4-\tilde{\lambda}_{m_0+1}}$.

Now we look at $r^0$ and $\tilde{r}^0$. By using Lemma \ref{lemma31} and Lemma \ref{lemma32}, noticing that $\lambda_{m_0+1}=\frac{5+\sqrt{17}}{2}\approx 4.562$, $\tilde{\lambda}_{m_0+1}=\frac{5-\sqrt{17}}{2}\approx0.438$, $\lambda_{m_0+2}=\frac{5-\sqrt{25-4\lambda_{m_0+1}}}{2}\approx1.201$ and $\tilde{\lambda}_{m_0+2}=\frac{5-\sqrt{25-4\tilde{\lambda}_{m_0+1}}}{2}\approx0.089$, it is easy to check that: if $\theta\in(-1,+\infty)$, then $r\in(-\infty,-1)\cup(4-\lambda_{m_0+1},+\infty)\cup\{\infty\}$, $\tilde{r}\in(-1,4-\tilde{\lambda}_{m_0+1})$, and then $r^0\notin[\frac{1}{4-\lambda_{m_0+2}},4-\lambda_{m_0+2}]$, $\tilde{r}^0\in(\frac{1}{4-\tilde{\lambda}_{m_0+2}},4-\tilde{\lambda}_{m_0+2})$; if $\theta\in(-\infty,-1)$, then $r\in(-1,4-\lambda_{m_0+1})$, $\tilde{r}\in(-\infty,-1)\cup(4-\tilde{\lambda}_{m_0+1},+\infty)\cup\{\infty\}$, and then $r^0\in(\frac{1}{4-\lambda_{m_0+2}},4-\lambda_{m_0+2})$, $\tilde{r}^0\notin[\frac{1}{4-\tilde{\lambda}_{m_0+2}},4-\tilde{\lambda}_{m_0+2}]$; if $\theta=-1$, then $r=4-\lambda_{m_0+1}$, $\tilde{r}=4-\tilde{\lambda}_{m_0+1}$, and then $r^0=4-\lambda_{m_0+2}$, $\tilde{r}^0=4-\tilde{\lambda}_{m_0+2}$; and if $\theta=\infty$, then $r=\tilde{r}=-1$, and then $r^0=\frac{1}{4-\lambda_{m_0+2}}$, $\tilde{r}^0=\frac{1}{4-\tilde{\lambda}_{m_0+2}}.$

The argument is similar for $r^1$ and $\tilde{r}^1$ by instead considering $\theta\in(-\infty,1)$, $\theta\in(1,+\infty)$, $\theta=1$ and $\theta=\infty$ separately. Thus the conclusion is true for $\lambda_{m_0}=2$.

Case 2: $\lambda_{m_0}=5$. In this case, $6-6\lambda_{m_0+1}+\lambda_{m_0+1}^2=1-\lambda_{m_0+1}$, then by (\ref{r}), 
\begin{equation*}
r=\frac{(1-\lambda_{m_0+1})u(p_1)-(4-\lambda_{m_0+1})u(p_0)-2u(p_2)}{(4-\lambda_{m_0+1})u(p_1)-(1-\lambda_{m_0+1})u(p_0)+2u(p_2)}
\end{equation*}
and
\begin{equation*}
\tilde{r}=\frac{(\lambda_{m_0+1}-4)u(p_1)-(\lambda_{m_0+1}-1)u(p_0)-2u(p_2)}{(\lambda_{m_0+1}-1)u(p_1)-(\lambda_{m_0+1}-4)u(p_0)+2u(p_2)}=\frac{1}{r}.
\end{equation*}

Now we look at $r^0$ and $\tilde{r}^0$. Noticing that $\lambda_{m_0+1}=\frac{5+\sqrt{5}}{2}\approx3.618$, $\tilde{\lambda}_{m_0+1}=\frac{5-\sqrt{5}}{2}\approx 1.382$, $\lambda_{m_0+2}\approx0.878$ and $\tilde{\lambda}_{m_0+2}\approx0.294$, by Lemma \ref{lemma31}, Lemma \ref{lemma32} and the fact that $4-\lambda_{m_0+1}=\frac{1}{\lambda_{m_0+1}-1}=\frac{1}{4-\tilde{\lambda}_{m_0+1}}$, it is easy to check that: if $r\in(-1,4-\lambda_{m_0+1})$, then $\tilde{r}=\frac{1}{r}\in(-\infty,-1)\cup(4-\tilde{\lambda}_{m_0+1},+\infty)\cup\{\infty\}$, and then $r^0\in(\frac{1}{4-\lambda_{m_0+2}},4-\lambda_{m_0+2})$, $\tilde{r}^0\notin[\frac{1}{4-\tilde{\lambda}_{m_0+2}},4-\tilde{\lambda}_{m_0+2}]$; if $r\in (-\infty,-1)\cup(4-\lambda_{m_0+1},+\infty)\cup\{\infty\}$, then $\tilde{r}\in(-1,4-\tilde{\lambda}_{m_0+1})$, and then $r^0\notin[\frac{1}{4-\lambda_{m_0+2}},4-\lambda_{m_0+2}]$, $\tilde{r}^0\in(\frac{1}{4-\tilde{\lambda}_{m_0+2}}, 4-\tilde{\lambda}_{m_0+2})$; if $r=-1$, then $\tilde{r}=-1$, and then $r^0=\frac{1}{4-\lambda_{m_0+2}}$, $\tilde{r}^0=\frac{1}{4-\tilde{\lambda}_{m_0+2}}$; and if $r=4-\lambda_{m_0+1}$, then $\tilde{r}=4-\tilde{\lambda}_{m_0+1}$, and then $r^0=4-\lambda_{m_0+2}$, $\tilde{r}^0=4-\tilde{\lambda}_{m_0+2}$.

The argument is similar for $r^1$ and $\tilde{r}^1$ by instead considering $r\in(-1,\frac{1}{4-\lambda_{m_0+1}})$, $r\in(-\infty,-1)\cup(\frac{1}{4-\lambda_{m_0+1}},+\infty)\cup\{\infty\}$, $r=-1$ and $r=\frac{1}{4-\lambda_{m_0+1}}$ separately. So the conclusion is true for $\lambda_{m_0}=5$.

Thus we have proved the lemma.
\end{proof}

As in Section 1, from now on, we use $N^E(u)$ and $N^E(\tilde{u})$ to denote the number of local extrema of $u$ and $\tilde{u}$ restricted to $E$, respectively.

\begin{lemma}\label{lemma35}
Assume $m_0\leq m<m_1-1$ ($m_0<m<m_1-1$ if $\lambda_{m_0}=6$).  We have

(1) if $r\neq4-\lambda_{m+1}$ and $\frac{1}{4-\lambda_{m+1}}$, then $N^E(u)+N^E(\tilde{u})=2^{m_1-m-1}$;

(2) if $r=4-\lambda_{m+1}$ or $\frac{1}{4-\lambda_{m+1}}$, then $N^E(u)+N^E(\tilde{u})=2^{m_1-m-1}-1$.
\end{lemma}

\begin{proof}
(1) Suppose $r\neq4-\lambda_{m+1}$ and $\frac{1}{4-\lambda_{m+1}}$. Let $\tau$ be any word with length $|\tau|=m_1-m-1$. Then $E_\tau$ is a $(m_1-1)$-edge contained in $E$.
 If $r^\tau\neq 4-\lambda_{m_1}$ and $\frac{1}{4-\lambda_{m_1}}$, then by Lemma \ref{lemma34} and Theorem \ref{thm2}, there must be one of $u|_{E^\tau}$ and $\tilde{u}|_{E^\tau}$ having a local extremum. If $r^\tau=4-\lambda_{m_1}$, then by Lemma \ref{lemma31} and Lemma \ref{lemma32}, $\tau\neq 0^{m_1-m-1}$, and thus we could write $\tau=\tau'10^k$ for some word $\tau'$ and integer $k\geq 0$. Then by Corollary \ref{cor31}, $p_0^\tau$ is an extreme point of $u$ in $E$, and $r^{\tau'01^k}=\frac{1}{4-\lambda_{m_1}}$. Furthermore, by Lemma \ref{lemma34}, we also have $\tilde{r}^\tau=4-\tilde{\lambda}_{m_1}$, $\tilde{r}^{\tau'01^k}=\frac{1}{4-\tilde{\lambda}_{m_1}}$, and $p_{0}^\tau$ is also an extreme point of $\tilde{u}$ in $E$. The situation is similar if $r^\tau=\frac{1}{4-\lambda_{m_1}}$. Sum up the extreme points provided by all words $\tau$ with $|\tau|=m_1-m-1$, we then have $N^E(u)+N^E(\tilde{u})=2^{m_1-m-1}$.
 
 (2) When $r=4-\lambda_{m+1}$ or $\frac{1}{4-\lambda_{m+1}}$, we only need to notice that there exists an exception of $\tau=0^{m_1-m-1}$ when $r=4-\lambda_{m+1}$, or $\tau=1^{m_1-m-1}$ when $r=\frac{1}{4-\lambda_{m+1}}$, which provides no extreme point.  So we have $N^E(u)+N^E(\tilde{u})=2^{m_1-m-1}-1$ in this case.
\end{proof}

Now we come to the proof of Theorem \ref{thm3}.

\begin{proof}[Proof of Theorem \ref{thm3}]

For $n\geq 0$, we write $n$ in  binary form, 
\begin{equation*}
n=\sum_{j=0}^\infty\delta_j^{(n)}2^j, \quad\delta_j^{(n)}=0 \text{ or } 1,
\end{equation*}
 with all but a finite number of $\delta_j^{(n)}$ equal to $0$. Let $\{\varepsilon^{(n)}_{m_0+i}\}_{i\geq 1}$ be a  $\varepsilon$-sequence defined as
 \begin{equation}\label{eql}
   \varepsilon^{(n)}_{m_0+i}=
      \begin{cases}
     \varepsilon_{m_0+i},\quad &\text{if } m_0+i\leq m,\\
(-1)^{1+\delta_{m_0+i-m-1}^{(n)}+\delta_{m_0+i-m}^{(n)}}, \quad &\text{if } m_0+i>m,\\
      \end{cases}
      \end{equation}
 and $\{\lambda^{(n)}_{m_0+i}\}_{i\geq 0}$ be the associated sequence of graph eigenvalues initialed from $\lambda_{m_0}$.  It is easy to check that $\psi_n$ is the associated eigenfunction of $\{\lambda^{(n)}_{m_0+i}\}_{i\geq 0}$ satisfying $\psi_n|_{V_m}=u|_{V_m}$.   In fact, if we write $\lambda^{(n)}=\frac{3}{2}\lim_{i\rightarrow\infty}5^{m_0+i}\lambda_{m_0+i}^{(n)}$, then it is easy to check that $\lambda^{(n)}<\lambda^{(n')}$ whenever $n<n'$. 

  For $n\neq 0$, define $t^{(n)}=\max\{i| \varepsilon_{m_0+i}^{(n)}=1, i\geq 1\}$, then $m_0+t^{(n)}+1$ is the level of fixation of $\psi_n$. Furthermore, from (\ref{eql}), we know that $\delta_{t^{(n)}+m_0-m-1}^{(n)}=1$ and $2^{t^{(n)}+m_0-m-1}\leq n<2^{t^{(n)}+m_0-m}$.  Since $m<m_0+t^{(n)}$, there exists an integer $n'\geq 0$ such that $\psi_{n'}=\tilde{\psi}_n$, and the associated $\varepsilon$-sequence of $\psi_{n'}$ should be
  \begin{equation*}
   \varepsilon^{(n')}_{m_0+i}=
      \begin{cases}
     \varepsilon^{(n)}_{m_0+i},\quad &\text{if } i\neq t^{(n)},\\
-1, \quad &\text{if } i=t^{(n)}.\\
      \end{cases}
      \end{equation*}
      This gives that 
        \begin{equation*}
   \delta_{j}^{(n')}=
      \begin{cases}
1-\delta_j^{(n)},\quad &\text{if } j\leq t^{(n)}+m_0-m-1,\\
=\delta_j^{(n)}=0, \quad &\text{if } j>t^{(n)}+m_0-m-1,\\
      \end{cases}
      \end{equation*}
      and so $n'+n=2^{t^{(n)}+m_0-m}-1$. 
      
      So for any $n\neq 0$, by letting $q\geq 1$ be the unique integer such that $2^{q-1}\leq n< 2^q$, we have $n'=2^q-1-n$. In particular, if $n=1$, then $n'=0$.
      
      On the other hand, for $n\geq 0$, if we denote by $r_n=r^E(\psi_n)$, it is easy to check that $r_n$ equals either $r_0$ or $r_1$ since $\lambda_{m+1}^{(n)}$ equals either $\lambda_{m+1}^{(0)}$ or $\lambda_{m+1}^{(1)}$ and $\psi_n|_{V_m}=u|_{m}$. Then since $\psi_0=\tilde{\psi}_1$ and the level of fixation of $\psi_1$ is $m+2$, by using Lemma \ref{lemma32} and Lemma \ref{lemma34}, we have that for any $n\neq 0$, $r_n=\frac{1}{4-\lambda_{m+1}^{(n)}}$ or $4-\lambda_{m+1}^{(n)}$ if and only if $r_0=\frac{1}{4-\lambda_{m+1}^{(0)}}$ or $4-\lambda_{m+1}^{(0)}$.
      
Thus if $\frac{1}{4-\lambda^{(0)}_{m+1}}<r_0<4-\lambda^{(0)}_{m+1}$, then for any $n\neq 0$, $r_n\neq\frac{1}{4-\lambda_{m+1}^{(n)}}$ and $4-\lambda_{m+1}^{(n)}$. By Theorem \ref{thm2}, we have $N^E(\psi_0)=0$, and  by Lemma \ref{lemma35}, for any $n\neq 0$, $N^E(\psi_n)+N^E(\psi_{2^q-1-n})=2^q$ if $2^{q-1}\leq n<2^q$ for some $q\geq1$. Then by induction, we have $N^E(\psi_n)=2[\frac{n+1}{2}]$.

If $r_0=\frac{1}{4-\lambda^{(0)}_{m+1}}$ or $4-\lambda^{(0)}_{m+1}$, then  for $n\neq 0$, $r_n=\frac{1}{4-\lambda^{(n)}_{m+1}}$ or $4-\lambda^{(n)}_{m+1}$, $N^E(\psi_0)=0$, and $N^E(\psi_n)+N^E(\psi_{2^q-1-n})=2^q-1$ for $n$ satisfying $2^{q-1}\leq n< 2^q$. This gives that $N^E(\psi_n)=n$ by induction.

Similarly,  if $r_0\notin[\frac{1}{4-\lambda^{(0)}_{m+1}}, 4-\lambda^{(0)}_{m+1}]$, then $N^E(\psi_0)=1$ and $N^E(\psi_n)+N^E(\psi_{2^q-1-n})=2^q$ for $n$ satisfying $2^{q-1}\leq n< 2^q$. This gives that $N^E(\psi_n)=2[\frac{n}{2}]+1$.

Thus we have proved the theorem.
\end{proof}

\section{Proof of Theorem \ref{thm4}}

We deal with $\{\psi_n^{(2)}\}_{n\geq 0}$, $\{\psi_n^{(5)}\}_{n\geq 0}$ and $\{\psi_n^{(6)}\}_{n\geq 0}$ respectively. 

\textbf{1. The $\psi_n^{(2)}$ case.}  We use $x_0, x_1, x_2$ to denote the boundary vertices of $\mathcal{SG}$, and $y_0,y_1,y_2$ to denote the vertices in $V_1\setminus V_0$ with $y_i$ opposite $x_i$. See the values of $\psi_n^{(2)}$ at $x_i$, $y_i$ in Fig. 7. If we let $E=\overline{y_1y_0}$ (or $\overline{y_2y_1}$, $\overline{y_0y_2}$ by symmetry), by using (\ref{eq32}), it is easy to check that $\forall n\geq 0$, $\psi_n^{(2)}$  takes constant $1$ on $E$ as $\lambda_{1}^{(n)}=2$ and $\frac{2(4-\lambda_2^{(n)})}{(2-\lambda_2^{(n)})(5-\lambda_2^{(n)})}=1$.

\begin{figure}[h]
\begin{center}
\includegraphics[width=4.5cm]{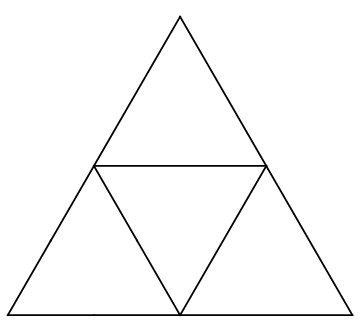}\hspace*{2cm}\includegraphics[width=4.5cm]{eigen25.jpg}
\begin{center}
\begin{picture}(0,0) \thicklines
\put(-162,12){$x_0$}\put(-96,9){$y_2$}\put(-58,74){$y_0$}
\put(-30,12){$x_1$}\put(-96,129){$x_2$}\put(-132,74){$y_1$}
\put(22,12){$0$}\put(90,9){$1$}\put(128,73){$1$}
\put(155,12){$0$}\put(90,129){$0$}\put(53,73){$1$}
\end{picture}
\textbf{\\Figure 7. The family $\psi_n^{(2)}$ with $\lambda_1^{(n)}=2$}
\end{center}\end{center}
\end{figure}

If we let $E=\overline{x_0y_2}$ (or $\overline{x_1y_2}$, etc. by symmetry), then from (\ref{r}) we have
\begin{equation*}
r^E(\psi_0^{(2)})=\frac{6-6\lambda_2^{(0)}+(\lambda_2^{(0)})^2-2}{4-\lambda_2^{(0)}+2}=\frac{1}{4-\lambda_2^{(0)}},\quad\big(\lambda_2^{(0)}=\frac{1}{2}(5-\sqrt{17})\big).
\end{equation*}
Then by Theorem \ref{thm3}, $N^E(\psi_n^{(2)})=n$.
Moreover, by symmetry, $\psi_n^{(2)}(y_2)$ must be an extremum, so the number of extrema of $\psi_n^{(2)}$ on $\overline{x_0x_1}$ is $2n+1$.

\textbf{2. The $\psi_n^{(5)}$ case.} See the values of $\psi_n^{(5)}$ at $x_i$, $y_i$ in Fig. 8. 

\begin{figure}[h]
\begin{center}
\includegraphics[width=4.5cm]{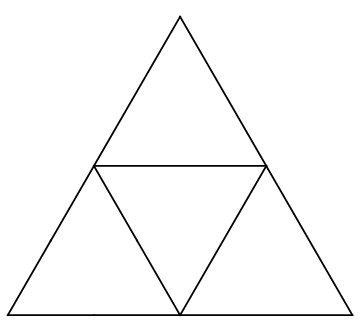}\hspace*{2cm}\includegraphics[width=4.5cm]{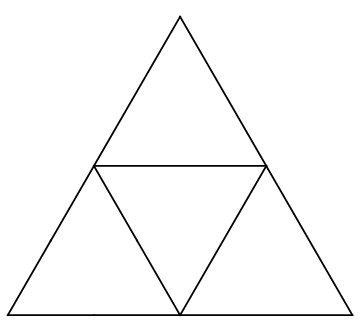}
\begin{center}
\begin{picture}(0,0) \thicklines
\put(-162,11){$x_0$}\put(-96,9){$y_2$}\put(-58,75){$y_0$}
\put(-30,11){$x_1$}\put(-96,129){$x_2$}\put(-133,75){$y_1$}
\put(22,12){$0$}\put(90,9){$0$}\put(128,73){$-1$}
\put(155,12){$0$}\put(90,129){$0$}\put(53,73){$1$}
\end{picture}
\textbf{\\Figure 8. The family  $\psi_n^{(5)}$ with $\lambda_1^{(n)}=5$}
\end{center}\end{center}
\end{figure}

For $E=\overline{x_0y_2}$ (or $\overline{x_1y_2}$), by (\ref{r}) we have $r^E(\psi_0^{(5)})=-1<\frac{1}{4-\lambda_2^{(0)}} \ \big(\lambda_2^{(0)}=\frac{1}{2}(5-\sqrt{5})\big)$. From Theorem \ref{thm3}, $N^E(\psi_n^{(5)})=2[\frac{n}{2}]+1$. Moreover, as $\psi_n^{(5)}$ is antisymmetric on $\overline{x_0x_1}$, $\psi_n^{(5)}(y_2)$ is not an extremum. So the total number of extrema of $\psi_n^{(5)}$ on $\overline{x_0x_1}$ is $4[\frac{n}{2}]+2$.

For $E=\overline{x_0y_1}$ (or $\overline{x_1y_0}$), $r^E(\psi_0^{(5)})=\frac{1-\lambda_2^{(0)}}{4-\lambda_2^{(0)}}<\frac{1}{4-\lambda_2^{(0)}}$, $N^E(\psi_n^{(5)})=2[\frac{n}{2}]+1$.

For $E=\overline{y_1x_2}$ (or $\overline{y_0x_2}$), $r^E(\psi_0^{(5)})=\frac{1}{4-\lambda_2^{(0)}}$, $N^E(\psi_n^{(5)})=n$.

Consider the point $y_1$. Like we did in Lemma \ref{lemma2}, we write $\psi_n^{(5)}$  into the summation of its symmetric part $\psi_{n,S}^{(5)}$ and antisymmetric part $\psi_{n,A}^{(5)}$ with respect to the line connecting $y_1$ and $x_1$, that is
\begin{eqnarray*}
&&\psi_{n,S}^{(5)}(x_0)=\psi_{n,S}^{(5)}(x_1)=\psi_{n,S}^{(5)}(x_2)=0\\
&&\psi_{n,S}^{(5)}(y_2)=\psi_{n,S}^{(5)}(y_0)=-\frac{1}{2},\quad
\psi_{n,S}^{(5)}(y_1)=1,\\
&&\psi_{n,A}^{(5)}(x_0)=\psi_{n,A}^{(5)}(x_1)=\psi_{n,A}^{(5)}(x_2)=\psi_{n,A}^{(5)}(y_1)=0,\\
&&\psi_{n,A}^{(5)}(y_2)=-\psi_{n,A}^{(5)}(y_0)=\frac{1}{2}.
\end{eqnarray*}
Then using the same argument, it is easy to check that $\psi_n^{(5)}(y_1)$ is not an extremum for any $n$, and thus the total number of extrema of $\psi_n^{(5)}$ on $\overline{x_0x_2}$ is $2[\frac{n}{2}]+n+1$.

In addition, for $E=\overline{y_2y_1}$ (or $\overline{y_2y_0}$), $r^E(\psi_0^{(5)})=\frac{1-\lambda_2^{(0)}}{4-\lambda_2^{(0)}}<\frac{1}{4-\lambda_2^{(0)}}$, $N^E(\psi_n^{(5)})=2[\frac{n}{2}]+1$. And for $E=\overline{y_1y_0}$, $\frac{1}{4-\lambda_2^{(0)}}<r^E(\psi_0^{(5)})=1<4-\lambda_2^{(0)}$, $N^E(\psi_n^{(5)})=2[\frac{n+1}{2}]$.

\textbf{3. The $\psi_n^{(6)}$ case.} 
To apply Theorem \ref{thm3}, we need to consider smaller edges in $\mathcal{SG}$. See the values of $\psi_n^{(6)}$ on $V_3$ in Fig. 9, noticing that $\lambda_3^{(n)}=3$, $\forall n\geq 0$.  

\begin{figure}[h]
\begin{center}
\includegraphics[width=6.5cm]{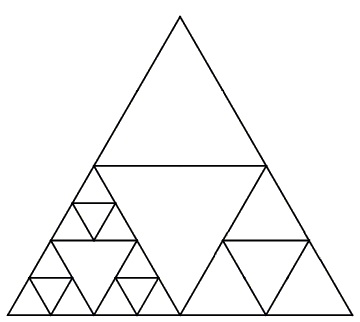}\hspace*{2cm}\includegraphics[width=6.5cm]{eigen6.jpg}
\begin{center}
\begin{picture}(0,0) \thicklines
\put(-219,11){$x_0$}\put(-126,11){$y_2$}\put(-172,11){$z_2$}
\put(-33,11){$x_1$}\put(-126,179){$x_2$}\put(-178,100){$y_1$}
\put(-196,11){$w_2$}\put(-149,11){$w_5$}\put(-200,56){$z_1$}
\put(-80,11){$z_5$}\put(-141,56){$z_0$}\put(-112,56){$z_4$}
\put(-52,56){$z_3$}\put(-74,100){$y_0$}\put(-172,51){$w_8$}
\put(-177,37){$w_0$}\put(-166,42){$w_4$}\put(-212,37){$w_1$}
\put(-132,37){$w_3$}\put(-190,77){$w_7$}\put(-154,77){$w_6$}
\put(24,11){$0$}\put(118,10){$2$}\put(70,11){$-1$}
\put(213,11){$0$}\put(118,177){$0$}\put(68,98){$0$}
\put(45,10){$-\frac12$}\put(95,10){$\frac12$}\put(45,56){$1$}
\put(160,11){$-1$}\put(101,56){$-1$}\put(128,56){$-1$}
\put(192,56){$1$}\put(169,98){$0$}\put(74,49){$0$}
\put(68,37){$0$}\put(73,40){$-1$}\put(33,37){$\frac{1}{2}$}
\put(113,37){$\frac12$}\put(56,77){$\frac12$}\put(89,77){$-\frac12$}

\end{picture}
\textbf{\\Figure 9. The family $\psi_n^{(6)}$ with $\lambda_2=6$}
\end{center}\end{center}
\end{figure}

It is easy to check that $\forall n\geq 0$, $\psi_n^{(6)}$ takes constant $-1$ on $\overline{z_2z_0}$ as $\psi_n^{(6)}(z_2)=\psi_n^{(6)}(z_0)=\psi_n^{(6)}(w_4)=-1$.

For $E=\overline{x_0w_2}$ (or $\overline{z_2w_2}$, $\overline{z_2w_5}$, $\overline{y_2w_5}$, $\overline{z_1w_0}$, $\overline{z_2w_0}$, etc. by symmetry, noticing the direction of the segments),
we have $r^E(\psi_0^{(6)})=4-\lambda_4^{(0)}\  \big(\lambda_4^{(0)}=\frac{1}{2}(5-\sqrt{13})\big)$, and thus $N^E(\psi_n^{(6)})=n$.

On the other hand, noticing that $\psi_n^{(6)}(z_2)<\psi_n^{(6)}(w_5)<\psi_n^{(6)}(y_2)$, $\psi_n^{(6)}(z_2)<\psi_n^{(6)}(w_2)<\psi_n^{(6)}(x_0)$, and $\frac{\psi_n^{(6)}(z_2)-\psi_n^{(6)}(w_5)}{\psi_n^{(6)}(w_5)-\psi_n^{(6)}(y_2)}=\frac{\psi_n^{(6)}(z_2)-\psi_n^{(6)}(w_2)}{\psi_n^{(6)}(w_2)-\psi_n^{(6)}(x_0)}$, we could find that $\psi_n^{(6)}$ attains an extremum at $z_2$, but not at $w_2$ or $w_5$. And from symmetry $\psi_n^{(6)}(y_2)$ must be an extremum. So the total number of extrema of $\psi_n^{(6)}$ on $\overline{x_0x_1}$ is $2(4n+1)+1=8n+3$.

In addition, one could also verify that for $\psi_n^{(6)}$ on $\overline{x_0y_1}$ (or $\overline{y_1y_2}$), the number is $4n+1$; and for $\psi_n^{(6)}$ on $\overline{z_1z_2}$ (or $\overline{z_1z_0}$), the number is $2n$.

Thus we have proved Theorem $\ref{thm4}$.

\end{document}